\documentclass[reqno,10pt]{amsart}
\usepackage{amssymb,amsmath,amsthm,amsfonts}

\usepackage{mathrsfs,dsfont, comment,mathscinet}

\usepackage{enumerate,esint}
\usepackage[left=2.8cm,right=2.8cm,top=2.8cm,bottom=2.8cm]{geometry}
\usepackage{mathtools}
\usepackage[pdftex]{color, graphicx}

\usepackage{epstopdf}
\usepackage[linktocpage=true,colorlinks=true,linkcolor=blue,citecolor=green]{hyperref}

\allowdisplaybreaks
\numberwithin{equation}{section}
\theoremstyle{plain}

\usepackage[english]{babel}
\usepackage{amsthm}
\usepackage{amsfonts, mathrsfs}
\usepackage{latexsym}
\usepackage{comment} % label interni
\usepackage{esint} % per il simbolo della media integrale
\usepackage{pdfsync}
\usepackage{enumitem}

\usepackage{caption}    %someone used \captionof below

\usepackage{cancel}     %cross out text, can remove for final version
\usepackage{ulem}       %strike through text with \sout
\usepackage[]{todonotes}
\newtheorem{thm}{Theorem}[section]
\newtheorem{prop}[thm]{Proposition}
\newtheorem{lemma}[thm]{Lemma}

\newtheorem{definition}[thm]{Definition}
\theoremstyle{definition}
\newtheorem{remark}[thm]{Remark}
\newtheorem{example}[thm]{Example}

\allowdisplaybreaks[2] 		%1-4 for finer control of page breaks in amsmath environments
\numberwithin{equation}{section}	%only count equations etc within a section

\newcommand{\ov}{\overline}
\newcommand{\wt}{\widetilde}
\newcommand{\e}{\varepsilon}
\newcommand{\ep}{\epsilon}
\newcommand{\vphi}{\varphi}
\def\loc{{\rm loc}}
\def\hom{{\rm hom}}
\def\eff{{\rm eff}}
\def\pot{{\rm pot}}
\def\sol{{\rm sol}}
\DeclareMathOperator*{\argmin}{arg\,min}
\newcommand{\tomega}{{\wt{\omega}}}

\newcommand{\N}{\mathbb{N}}

\newcommand{\R}{\mathbb{R}}

\newcommand{\domain}{D}

\def\div{{\operatorname{div}}}

\newcommand{\divv}{\nabla\cdot}
\newcommand{\curl}{\operatorname{curl}} 
\newcommand{\dx}{\,\mathrm{d}x}

\newcommand{\f}{\boldsymbol}
\newcommand{\m}{\,\mathrm{d}}				%use case: write \m\f{x}	for dx at the end of integral

\newcommand{\Id}{\operatorname{Id}}
\newcommand{\Ev}{\mathbb{E}}				%expected value
\def\P{{\mathbb{P}}}			        	%probability measure
\newcommand{\dP}{\,\mathrm{d}\P}
\newcommand{\dPo }{\,\mathrm{d}\P(\omega)}

\newcommand{\muz}{{\mu_0}}			    	%vacuum permeability
\newcommand{\bchi}{{\boldsymbol\chi}}		%disintegrated curl structure
\newcommand{\Rt}{\R^{3}}
\newcommand{\Rtt}{\R^{3\times 3}}
\newcommand{\St}{\mathbb{S}^2}		    	%sphere in R^3
\newcommand{\M}{\mathcal{M}}				%Manifold
\newcommand{\T}[1][ ]{{\mathcal{T}_{#1}}}	%Tangent space at #1
\newcommand{\bT}[1][ ]{{\boldsymbol{\mathcal{T}}_{#1}}}	%(Tangent space at #1)^3
\newcommand{\boldT}{\boldsymbol{\mathcal{T}}}
\newcommand{\F}{\mathcal{F}}				%micromagnetic energy
\newcommand{\Fhom}{\mathcal{F}_\hom}
\newcommand{\E}{\mathcal{E}}				%exchange energy
\newcommand{\K}{\mathcal{K}}				%DMI energy
\newcommand{\A}{\mathcal{A}}				%anisotropy energy
\newcommand{\Ahom}{\mathcal{A}_\hom}
\newcommand{\W}{\mathcal{W}}				%magnetostatic self-energy
\newcommand{\Whom}{\mathcal{W}_\hom}
\newcommand{\Ze}{\mathcal{Z}}				%Zeeman energy
\newcommand{\Zehom}{\mathcal{Z}_\hom}
\newcommand{\G}{\mathcal{G}}				%\F+\K
\newcommand{\Ghom}{\mathcal{G}_\hom}
\newcommand{\hd}{{h_d}}
\newcommand{\Ms}{M_{\mathrm{sat}}}
\newcommand{\Ltpot}{{L^2_{\mathrm{pot}}}}
\newcommand{\Ltsol}{{L^2_{\mathrm{sol}}}}

\newcommand{\nablaomega}{{\nabla_\omega}}
\newcommand{\Cic}{{C^\infty_{\rm c}}}
\newcommand{\BL}{{{\rm BL}^1}}
\newcommand{\wsts}{{\xrightharpoonup{\text{stoch 2-s}}}}    %weak stochastic 2-scale convergence
\newcommand{\sts}{{\xrightarrow{\text{stoch 2-s}}}}         %strong stochastic 2-scale convergence

\newcommand{\BBB}{\color{black}}

\title[Stochastic homogenization of micromagnetic energies]{Stochastic homogenization of micromagnetic energies and emergence of magnetic skyrmions}
\author[E. Davoli]{Elisa Davoli}
\address{Institute of Analysis and Scientific Computing, TU Wien, Wiedner Hauptstraße 8-10, 1040 Vienna, Austria}
\email{elisa.davoli@tuwien.ac.at}
\author[L. D'Elia]{Lorenza D'Elia}
\address{Institute of Analysis and Scientific Computing, TU Wien, Wiedner Hauptstraße 8-10, 1040 Vienna, Austria}
\email{lorenza.delia@tuwien.ac.at}
\author[J. Ingmanns]{Jonas Ingmanns}
\address{Institute of Science and Technology Austria (ISTA), Am Campus 1, 3400 Klosterneuburg, Austria}
\email{jonas.ingmanns@ist.ac.at}
\date{}
\begin{document}

\maketitle
%\tableofcontents

\begin{abstract}
We perform a stochastic-homogenization analysis for composite materials exhibiting a random microstructure.
Under the assumptions of stationarity and ergodicity, we characterize the Gamma-limit of a micromagnetic energy functional defined on magnetizations taking value in the unit sphere, and including both symmetric and antisymmetric exchange contributions. 
This Gamma-limit corresponds to a micromagnetic energy functional with homogeneous coefficients.  
    We provide explicit formulas for the effective magnetic properties of the composite material in terms of homogenization correctors.
    Additionally, the variational analysis of the two exchange energy terms is performed in the more general setting of functionals defined on manifold-valued maps with Sobolev regularity, in the case in which the target manifold is a bounded, orientable smooth surface with tubular neighborhood of uniform thickness.  
Eventually, we present an explicit characterization of minimizers of the effective exchange in the case of magnetic multilayers, providing quantitative evidence of Dzyaloshinskii's predictions on the emergence of helical structures in composite ferromagnetic materials with stochastic microstructure.
\end{abstract}\hspace{10pt}

\section{Introduction}
 Many key properties and applications of magnetic materials are strongly intertwined with the spatial distribution of magnetic moments inside the corresponding specimens. In addition to classical magnetic structures such as Weiss domains, where the magnetization is almost collinear and slowly-varying, and Bloch walls, forming thin transition layers and allowing the magnetization to rotate coherently from one magnetic structure to the other, magnetic skyrmions \cite{ferriani24, bogdanov8, fert2} have recently emerged as an extremely active field of research due to their potential as building blocks for innovative functional devices. Since skyrmions can be manipulated (written and deleted) individually on a magnetic stripe, in fact, these quasiparticles are regarded as possible carriers of information for future storage devices \cite{fert}. Named after the physicist T. Skyrme, skyrmions have been first predicted in magnetic crystals in \cite{djano}. They were then experimentally identified both in thin and ultrathin films as well as in multilayers.\\

The specific chirality of magnetic skyrmions, namely their lack of inversion symmetry \cite{fert2}, is determined by the antisymmetric Dzyaloshinskii-Moriya interaction (DMI, also called antisymmetric exchange)\cite{djano,fields28}. For an open bounded domain $D\subset \mathbb{R}^3$ representing a specimen of a single-crystal chiral magnet, assuming that $m$ satisfies a normalized saturation constraint
$|m(x)|=1$, the micro-magnetic energy is given by the functional (see \cite{brown})
\begin{align}
\F(m) &:= 
   	\frac{1}{2}\int_{D} a |\nabla m(x)|^2\dx 
   	+ \int_D \kappa \curl m(x)\cdot m(x)\dx
    - \frac{\mu_0}{2} \int_D h_d\left[m\chi_D\right](x)\cdot  m(x)\dx
    \notag\\&\quad    	
   	+\int_D \varphi\left(m(x)\right)\dx 
     -\mu_0 \int_D h_a(x)\cdot  m(x)\dx 
    \notag\\& =: 
    \E (m) + \K (m) + \W (m) + \A (m) + \Ze (m)\label{def:funcF-no-ep}
   \end{align}
In the expression above, $\E(m)$ and $\K(m)$ represent the exchange energy and the bulk DMI, encoding energetic contributions due to spatial changes of the magnetization and to asymmetries in the crystalline structure of the material, respectively. The parameters $a>0$ and $\kappa\in \mathbb{R}$ are material length-scales tuning the strength of very short-range interactions, and of chirality effects. 
The term $\W(m)$ in \eqref{def:funcF-no-ep} is the stray-field energy. It is a nonlocal energy contribution encoding the effects of the magnetic field $h_d[m\chi_D]$ induced by the magnetization $m$, and favoring configurations involving solenoidal magnetic fields. Note that $m\chi_{D}$ denotes an extension of $m$ to the whole $\Rt$ by setting it equal to zero outside of the set $D$. These two latter quantities are related through the magnetostatic Maxwell equation
$${\rm div}\,(-\mu_0 h_d[m\chi_D]+m\chi_D)=0\quad\text{in }\mathbb{R}^3,$$
where $\mu_0$ is the vacuum permeability. 
The contribution $\A(m)$ is the magnetocrystalline anisotropy, whose density $\varphi:\mathbb{S}^2\to [0,+\infty)$ encodes the presence of easy axes in the micromagnetic specimen, namely, of preferred magnetization directions (here $\mathbb{S}^2$ denotes, as usual, the unit sphere in $\R^3$). Eventually, $\Ze(m)$ is the Zeeman energy, modeling the tendency of the magnetization to align with the externally applied magnetic field $h_a$. A mathematical analysis of micromagnetics models incorporating DMI terms can be found in \cite{Li-Melcher35, Melcher38, Muratov42}. We also refer to the works \cite{cicalese13, cicalese14, cicalese15} for the discrete-to-continuum setting, as well as to \cite{ginster-zwicknagl23} for a study of associated energy scaling laws, and to \cite{dirk, davoli-difratta-praetorius-ruggeri} for related numerical and dimension-reduction results.
\\

The goal of this paper is to advance the mathematical modeling of magnetic skyrmions by analyzing the interplay of stochastic microstructures and chirality. In a recent work \cite{DDF20}, jointly with G. Di Fratta, one of the authors has  undertaken a periodic homogenization analysis for a micromagnetic energy functional involving a DMI bulk energy term. In particular, the results in \cite{DDF20} provided a first quantitative counterpart to the theoretical predictions in the seminal work by Dzyaloshinskii on the existence of helicoidal textures as the result of possible instabilities of ferromagnetic structures under small relativistic spin-lattice or spin-spin interactions \cite{djano, dzyalo20, dzyalo21, bak6}. \\

In the present manuscript we move radically beyond the even distribution of material heterogeneities assumed in \cite{DDF20} to identify effective theories for composite chiral magnetic materials with a microstructure encompassing random effects. The novelty of our contribution is threefold. First, we provide a stochastic homogenization analysis in the general setting of manifold-valued Sobolev spaces, having as a corollary the aforementioned application in micromagnetics. To the Authors' knowledge, this is the first mathematical study in this direction. Second, we delineate the framework for stochastic two-scale convergence in Beppo Levi spaces. Third, we present an explicit characterization of minimizers of the limiting exchange energies in the case of stochastic multilayers providing a further quantitative evidence, this time in an aperiodic setting, to Dzyaloshinskii predictions and to the experimental observations in \cite{yu52,fert2,chen12}, as well as a stochastic counterpart to \cite{DDF20}.\\ 

Before describing our findings in detail, we briefly review the mathematical literature on homogenization in micromagnetics.
Among the many contributions, we refer the Reader to the following papers and to the references therein: the setting of ferromagnetic laminates has been studied in \cite{haddar.joly}, whereas that of  perforated domains in \cite{santugini}. An analysis relying on the notion of $\mathcal{A}$-quasiconvexity has been carried out in \cite{pisante}. We refer to \cite{AdBMN21,ADF15} for two extensive homogenization results in the periodic and stochastic case. Periodic homogenization of chiral magnetic materials was tackled in the aforementioned \cite{DDF20}. \\

%In order to describe our results we need to introduce some notation. Let $\M$ be a bounded, $C^2$ orientable hypersurface of $\R^3$ which admits a tubular neighborhood of uniform thickness -- in micromagnetics applications usually $\M=\St$.  

When considering a composite ferromagnetic body, it is important to keep track of the local interactions of grains with different magnetic properties at their interface \cite{acerbi-fonseca-mingione}. In particular, under the assumption of strong coupling conditions, meaning that the direction $m$ of the magnetization does not jump through an interface and only the magnitude is allowed to be discontinuous, homogenization problems are formulated by considering point- and microstructure-dependent material parameters, and for magnetizations $m$ taking value in the unit sphere. Assuming that $D$ is a specimen of a multi-crystal ferromagnet with random microstructure, we thus consider the family of energy functionals $(\F_\e)_{\e>0}$ with $\F_\e\colon H^1(D; \hspace{0.03cm} \St)\to \R$ given by
   \begin{align}
   	\F_\e(m) 
   	&:= 
   	\frac{1}{2}         \int_\domain a\left(\frac{x}{\e},\tomega\right) |\nabla m(x)|^2\dx 
   	+                   \int_\domain \kappa\left(\frac{x}{\e},\tomega\right) \curl m(x)\cdot m(x)\dx
   	\notag\\&\quad
        - \frac{\mu_0}{2}   \int_\domain h_d\left[\Ms\left(\frac{x}{\e},\tomega\right)m\chi_\domain\right](x)\cdot \Ms\left(\frac{x}{\e},\tomega\right) m(x)\dx
        \notag\\&\quad    	
   	+                   \int_\domain \vphi\left(\frac{x}{\e},m(x),\tomega\right)\dx 
        - \mu_0             \int_\domain h_a\cdot \Ms\left(\frac{x}{\e},\tomega\right) m(x)\dx 
        \notag\\& 
        =                   \E_\e (m) + \K_\e (m) + \W_\e (m) + \A_\e (m) + \Ze_\e (m)\label{def:funcF}
   \end{align}
for a set of stationary and ergodic random parameters $a,\kappa, \Ms\colon\Rt\rightarrow \R$ and  $\vphi\colon\Rt\rightarrow C^{0,1}(\St)$ on an appropriate probability space $(\Omega,\Sigma,\P)$ with $\tomega\in\Omega$.
The specific stochastic framework and  assumptions on the parameters will be described in Section \ref{subs:assumptions}. We only briefly mention here that requiring stationarity amounts, roughly speaking, to imposing a sort of stochastic periodicity (or periodicity in law, cf \cite{dalmaso.modica}) guaranteeing that any given microstructure and all its translations occur with equal probability. Ergodicity, instead, enforces that the behavior of the effective material is independent from the specific stochastic realization, see \cite{dalmaso.modica2}.

Note that the DMI term can be rewritten as 
\begin{equation}\label{eq:intro:DMI}
    \K_\e (m)
    =
    \int_\domain \kappa\left(\frac{x}{\e},\tomega\right) \curl m(x)\cdot m(x)\dx
    =
    \int_\domain \nabla m(x):\kappa\left(\frac{x}{\e},\tomega\right)\bchi(m(x))\dx
\end{equation}
where $A:B=\sum_{i,j}A_{ij}B_{ij}$ for $A,B\in\Rtt$ and 
\begin{equation}\label{eq:intro:def:bchi}
   	\bchi\colon s\in\Rt\mapsto (e_1\times s, e_2\times s, e_3\times s)^\intercal\in \R^{3\times 3}_{\rm skew}.
\end{equation}

Let $\Id$ denote the identity matrix in $\R^{3\times 3}$, and for a random variable $f$, let $\Ev[f]:=\int_\Omega f(\omega)\dP$ be its expected value. We present below a slightly simplified statement of our first main result. We refer to Theorem \ref{mainthm} for its precise formulation.

\begin{thm}\label{mainthm-intro}
Let $D\subset \R^3$ be an open, bounded set with Lipschitz boundary. 
Then, under the assumption of stochasticity and ergodicity of all material parameters, for almost every $\tomega\in\Omega$ the family $(\F_\e)_{\e>0}$ is equi-coercive and $\Gamma$-converges with respect to the weak topology on $H^1(D;\hspace{0.03cm}\St)$ to the energy functional $\Fhom$, which corresponds to the micromagnetic energy of a homogeneous material with a-priori anisotropic exchange parameter and DMI constant, 
\begin{align}\notag%\label{eq_Fhom}
    \Fhom
    &=
    \E_\eff+\K_\eff+\W_\eff+\A_\eff+\Ze_\eff,
\end{align}
where, in the sense of Remark \ref{rk:correctors}, denoting by $\nabla\Phi_a,\nabla\Phi_\kappa,\nabla\Phi_M$ the solutions in $L^2(\Omega,L^2_{\rm loc}(\Rt))$ of the corrector-type equations
\begin{equation}\label{eq:mt:correctors-intro}
    \divv(a\nabla\Phi_a + a\Id)=0,
    \qquad
    \divv(a\nabla\Phi_\kappa + \kappa\Id)=0,
    \qquad
    \divv(\nabla\Phi_M + \Ms\Id)=0,
\end{equation}
there holds
\begin{align*}
    \E_\eff  &= \frac{1}{2}\int_\domain \nabla m:\Ev\left[a\Id-(\nabla\Phi_a)^\intercal a\nabla\Phi_a\right]\nabla m\dx,\\
    \K_\eff  &= -\int_\domain \nabla m :\Ev\left[\kappa\Id-(\nabla\Phi_a)^\intercal a\nabla\Phi_\kappa\right]\bchi(m)\dx,\\
    \W_\eff  &= -\frac{\mu_0}{2} \int_\domain h_d\left[\Ev\left[\Ms\right]m\chi_\domain\right]\cdot \Ev\left[\Ms\right] m\dx,\\
    \A_\eff  &= \int_\domain\left( \Ev\left[\vphi\right](m)
                                    - \frac{1}{2}\bchi(m):\Ev\left[(\nabla\Phi_\kappa)^\intercal a\nabla\Phi_\kappa\right]\bchi(m)
                                    + \frac{\mu_0}{2}m\cdot\Ev\left[(\nabla\Phi_M)^\intercal\nabla\Phi_M\right]m
                \right)\dx\\
    \Ze_\eff &= -\mu_0\int_\domain h_a\cdot \Ev\left[\Ms\right]m \dx.
\end{align*}
\end{thm}
We stress that the effective material identified via our homogenization procedure reduces, in the periodic case, to the one found in \cite{DDF20}. In this latter setting, the corrector equations are replaced by  corresponding cell problems and the expected values of the limiting random variables by the averages of the corresponding periodic quantities in their periodicity cell.
We also point out that the different energy contributions denoted with the pedix ``eff" in Theorem \ref{mainthm-intro} are not the $\Gamma$-limit of the corresponding $\e$-energies (which we will denote by ``hom" instead) but have rather been rearranged in order to highlight the role of correctors in the limiting problem, as well as the meaning of the various energy terms.

 The proof of Theorem \ref{mainthm} is based on the notion of quenched stochastic two-scale convergence. Whereas the notion of two-scale convergence in the periodic case \cite{lukkassen.nguetseng.wall,allaire, nguetseng} is uniquely defined, in the stochastic setting two alternative notions have been introduced, corresponding to two different choices for the topology in which the limiting description is identified: when working with stochastic two-scale convergence in the mean \cite{papanikolau.varadhan, andrews.wright, bourgeat.luckhaus.mikelic}, the relevant fields are integrated with respect to the probability space (see also \cite{neukamm.varga,heida.neukamm.varga} for a corresponding notion of stochastic unfolding), whereas quenched stochastic limits \cite{AdBMN21, ZP06} are taken pointwise for almost every realization. We refer to \cite{heida.neukamm.varga} for a comparison between the two convergences. We will argue here with the quenched variant, henceforth simply referred to, for shortness, as stochastic two-scale convergence.

The key ingredient for showing Theorem \ref{mainthm-intro} consists in proving the $\Gamma$-convergence of $\G_\e\coloneqq\E_\e+\K_\e$, for the remaining energy contributions can be treated as continuous perturbations. Studying the asymptotic behavior of $\G_\e$ is of independent mathematical interest, because it represents a first stochastic homogenization result for Dirichlet-type energies in manifold-valued Sobolev spaces. We therefore prove a slightly more general result for functionals defined on maps taking value in a general bounded, orientable, $C^2$ hypersurface of $\Rt$ admitting a tubular neighborhood of uniform thickness, cf. Theorem \ref{thm:GeGammaConv}. We refer to \cite{ADF15, DDF20, dacorogna.fonseca, babadjian.millot} for periodic homogenization problems in manifold-valued Sobolev spaces, as well as to \cite{berlyand.sandier.serfaty} and the references therein for a overview on stochastic homogenization in classical Sobolev spaces for convex and non-convex integral functionals.

The $\Gamma$-convergence of the magnetostatic self-energy is characterized in Proposition \ref{prop:intro:Whom}, whereas anisotropic energy and Zeeman contributions are studied in Proposition \ref{prop:intro:AhomZehom}. In particular, as a by-product of our analysis, we provide an extension of the theory of two-scale convergence in Beppo-Levi spaces in \cite{ADF15} to the stochastic setting, cf. Subsection \ref{section:demagnefield}. 

Eventually, we specify our analysis to the case in which our micromagnetic specimen is a multilayer with random microstructure. In this latter framework, we provide an explicit characterization of the minimizers of $\G_{\rm hom}\coloneqq{\Gamma\text{-}\lim}_{\e\to 0} (\E_\e+\K_e)$  and show the emergence of chiral structures.
A simplified statement of our second main result reads as follows. We refer to Proposition \ref{prop:lami:Ghom} and Lemma \ref{lem:lami:minimizers} for its precise formulation.

\begin{thm}
    \label{mainthm2-intro}
    If $\Ev[\kappa]=0$, the only minimizers of $\Ghom$  are given by
     \begin{equation*}
         \notag
         m_\ast(x):=\cos(\theta(x\cdot e_3))e_1 + \sin(\theta(x\cdot e_3))e_2, \quad \theta(t):= \theta_0+\Ev\left[\frac{\kappa}{a}\right]t,
     \end{equation*}
for every $t\in\R$, with $\theta_0\in\R$ arbitrary.
\end{thm}

The paper is organized as follows. 
In Section \ref{sec:Setting2scale} we describe the stochastic framework and specify the assumptions on the random parameters as well as the anisotropy energy density. We conclude this section providing the precise statements of our main results. 
In Section \ref{section:sto2scalecon} we recall the concept of stochastic two-scale convergence and show the existence of the corrector quantities given in \eqref{eq:mt:correctors}.
Section \ref{section: twoscale lim} is devoted to the characterization of the two-scale limits of $H^1(D;\hspace{0.03cm}\M)$-maps. 
In Section \ref{section:GammaresultExDMI}, we obtain the $\Gamma$-convergence of $\G_\e$ (Theorem \ref{thm:GeGammaConv}).
In particular, the limiting density is characterized in Subsection \ref{subsect:homDensity}, whereas the  liminf and limsup inequalities are proven in Subsections \ref{subsect:liminfine} and \ref{subsect:limsupine}, respectively. 
Section \ref{section:ProofMainResult} contains the rest of the proof of Theorem \ref{mainthm-intro} (for the precise assumptions, Theorem \ref{mainthm}): 
adapting the strategy in \cite{ADF15} to the stochastic setting, we prove the convergence of $\W_\e$ (Proposition \ref{prop:intro:Whom}) in Subsection \ref{section:demagnefield} and the convergence of $\A_\e$ and $\Ze_\e$ (Proposition \ref{prop:intro:AhomZehom}) in Subsection \ref{section:Aniso+Zeeman}.
The equi-coerciveness of the micromagnetic functionals is shown in Subsection \ref{section:equicoerc}.
Finally, Section \ref{section:multilayers} is devoted to obtain Theorem \ref{mainthm2-intro}: we calculate the effective material properties if the microstructure is given by laminates and thus obtain a characterization of the minimizers.

\section{Setting of the problem and main results}
\label{sec:Setting2scale}
In this section we collect our main assumptions and state our main results. For simplicity,  we will formulate all results in $\R^3$. We stress, however, that they hold more generally in $\R^d$.

\paragraph{\textbf{Notation}}
In what follows, we denote by $\Id$ the identity matrix in $\Rtt$. For  
$A\in \R^{3\times 3}$, we indicate by $A^\intercal$ its transpose.  Given a random variable $f$, the expected value of $f$ is denoted by $\Ev[f]$. We will use standard notation for Lebesgue and Sobolev spaces. We will sometimes omit the target space whenever this is immediately clear from the context. We will often denote by $C>0$ a generic constant, whose value in a formula might change from line to line. Diagonal matrices in $\Rtt$ with entries $d_1,d_2,d_3$ will be denoted by $\operatorname{diag}(d_1,d_2,d_3)$.

\subsection{Stochastic framework and parameter assumptions}
\label{subs:assumptions}
Let $(\Omega, \Sigma, \P)$ be a probability space endowed with a {\it dynamical system} $(T_x)_{x\in\Rt}$, i.e., a family of measurable bijective mappings 
$T_x:\Omega\to\Omega$ satisfying the following proprieties
%\todos{J:@Lorenza: Removed the bijectivity for now, was not in Zhikov -- where did it come from?}
\begin{enumerate}[label=(T\arabic*)]
    \item\label{asT_SumGroupAction}
        $T_0={\rm Id}$ and $T_{x+y} = T_x\circ T_y$, for every $x,y \in\Rt$;
    \item\label{asT_PInvariance}
        $\mathbb{P}(T_x^{-1}(B))=\mathbb{P}(B)$ for all $x\in\Rt$ and $B\in\Sigma$;
    \item\label{asT_measurability} 
        the map $(x, \omega)\in\R^3\times\Omega\mapsto
        T_x\omega\in\Omega$ is measurable.
\end{enumerate} 
%Whenever these conditions are met, the family $(T_x)_{x\in\Rt}$ is called a {\it dynamical system} on the probability space $(\Omega, \Sigma, \P)$.
%Additionally assume the following.
We also assume that $\Omega$ is a compact metric space, $\Sigma$ is a complete $\sigma$-algebra on $\Omega$, and $(x,\omega)\mapsto T_x\omega$ as a map from $\Rt\times\Omega$ to $\Omega$ is continuous. We will work under the assumption that the dynamical system $T_x$ is {\it ergodic}, namely:
\begin{enumerate}[resume*]
    \item\label{asT_Ergodicity}
         The only measurable sets $A\in\Sigma$ which are translation invariant, i.e.,  are such that (up to a null subset) $T_x A=A$ for every $x\in\R^d$, satisfy $\mathbb{P}(A)\in\{0,1\}$. 
    %\item\label{asT_compMetricSp}       
\end{enumerate}
\begin{remark}
    \rm 
    The ergodicity property can be equivalently formulated by requiring that the only measurable functions $f:\Omega\to \R$ such that $f(\omega) = f(T_x\omega)$ for every $x\in\Rt$ and $\P$-a.e. $\omega\in\Omega$ are those coinciding $\P$-almost surely with a constant.
\end{remark}

\begin{remark}
    \rm
    A range of settings and examples, where the above assumptions on the probability space and the dynamical system are satisfied, can be found in \cite[Section 1.1]{ZP06}.
\end{remark}

    For $1\leq p\leq\infty$ we denote by $L^p(\Omega)$ the $L^p$-space on $(\Omega, \Sigma, \P)$, and, similarly, we indicate by $L^p(\Rt\times\Omega)$ the $L^p$-space on the product measure space.

A random field $(x,\omega)\in\Rt\times\Omega\mapsto\tilde{f}(x, \omega)$  is said to be {\it stationary} if there exists a measurable function $f$ on $\Omega$ such that $\tilde{f}(x, \omega) = f(T_x\omega)$. 
Hence, a random variable $f: \Rt\times\Omega\to\R$ is {\it stationary ergodic} if $f$ is stationary and the underlying dynamical system $T_x$ is ergodic. \\
  %  Thus, e.g. \ref{asPara_Ex} states that the exchange parameters are given by $a_\e(x,\omega)=\tilde{a}(x/\e,\omega)$ for a positive stationary random field $\tilde{a}$ uniformly bounded from above and away from zero. 
We assume that all the material-dependant parameters and the anisotropy energy density are stationary ergodic random variables. With a slight abuse of notation, we will use the same letter for each of these random fields as well as for their corresponding measurable functions.
To be precise,
\begin{enumerate}[label=(P\arabic*)]
   	\item\label{asPara_Ex}
   	    the exchange coefficient is given by $a(x,\omega) = a(T_{x}\omega)$ with $a\in L^\infty(\Omega)$ bounded from below and above by two positive constants $c_{\rm ex}, C_{\rm ex}$, i.e.,  $0<c_{\rm ex}\leq a(\omega)\leq C_{\rm ex}$ for all $\omega\in\Omega$;
   	\item\label{asPara_DMI}
   	    the DMI constant is given by $\kappa(x,\omega) = \kappa(T_{x}\omega)$,  for some $\kappa\in L^\infty(\Omega)$ bounded by a positive  constant $C_{\rm DMI}$, i.e. $|\kappa(\omega)|\leq C_{\rm DMI}$ for all $\omega\in\Omega$;
   	\item\label{asPara_Msat}
   	    the local saturation magnetization is given by $\Ms(x,\omega) = \Ms(T_{x}\omega)$ with $\Ms\in L^\infty(\Omega)$ positive and bounded from above, i.e.,  $0\leq\Ms(\omega)\leq C_{\rm sat}$ for some $C_{\rm sat}>0$ and all $\omega\in\Omega$;
   	\item\label{asPara_aniso}
   	    the anisotropy energy density $\varphi\colon\R^3\times\Omega\times \St\to \R^{+}$ is given by $\vphi\left(x,s,\omega\right)=\vphi(T_{x}\omega,s)$, where $\vphi(\cdot,s): \Omega\to\R^{+}\in L^\infty(\Omega)$ for every $s\in\St$, and there exists a positive constant $L_{\rm an}>0$ such that 
   	    \begin{equation}
   	    \notag
   	    \underset{\omega\in\Omega}{\mbox{ess-sup}}\hspace{0.03cm} |\varphi(\omega, s_1)-\varphi(\omega, s_2)| \leq L_{\rm an} |s_1-s_2|, \quad \forall\hspace{0.1cm} s_1, s_2\in\St.
   	    \end{equation}
\end{enumerate}

\subsection{Statement of the main results}
With the notation introduced in Subsection \ref{subs:assumptions}, we are now in a position to provide the precise statement of our main result. Recall \eqref{eq:intro:def:bchi} and \eqref{def:funcF}.
\begin{thm}\label{mainthm}
Let $D$ be an open, bounded domain with  Lipschitz boundary. Assume that, with respect to a probability space $(\Omega,\Sigma,\P)$ as described in Section \ref{subs:assumptions}, the random parameters $a_\e,\kappa_\e,M_\e$ and the random anisotropy density $\vphi_\e$ satisfy \ref{asPara_Ex}--\ref{asPara_aniso}.\\
Then, for almost every $\tomega\in\Omega$ the family $(\F_\e)_{\e>0}$ is equi-coercive and $\Gamma$-converges with respect to the weak topology on $H^1(D;\hspace{0.03cm}\St)$ to the energy functional $\Fhom$, which corresponds to the micromagnetic energy of a homogeneous material with a-priori anisotropic exchange parameter and DMI constant, 
\begin{align}\notag%\label{eq_Fhom}
    \Fhom
    &=
    \E_\eff+\K_\eff+\W_\eff+\A_\eff+\Ze_\eff,
\end{align}
where, in the sense of Remark \ref{rk:correctors}, denoting by $\nabla\Phi_a,\nabla\Phi_\kappa,\nabla\Phi_M$ the solutions in $L^2(\Omega,L^2_{\rm loc}(\Rt))$ of the corrector-type equations
\begin{equation}\label{eq:mt:correctors}
    \divv(a\nabla\Phi_a + a\Id)=0,
    \qquad
    \divv(a\nabla\Phi_\kappa + \kappa\Id)=0,
    \qquad
    \divv(\nabla\Phi_M + \Ms\Id)=0,
\end{equation}
the effective energy components are given by
\begin{align*}
    \E_\eff  &= \frac{1}{2}\int_\domain \nabla m:\Ev\left[a\Id-(\nabla\Phi_a)^\intercal a\nabla\Phi_a\right]\nabla m\dx,\\
    \K_\eff  &= -\int_\domain \nabla m :\Ev\left[\kappa\Id-(\nabla\Phi_a)^\intercal a\nabla\Phi_\kappa\right]\bchi(m)\dx,\\
    \W_\eff  &= -\frac{\mu_0}{2} \int_\domain h_d\left[\Ev\left[\Ms\right]m\right]\cdot \Ev\left[\Ms\right] m\dx,\\
    \A_\eff  &= \int_\domain\left( \Ev\left[\vphi\right](m)
                                    - \frac{1}{2}\bchi(m):\Ev\left[(\nabla\Phi_\kappa)^\intercal a\nabla\Phi_\kappa\right]\bchi(m)
                                    + \frac{\mu_0}{2}m\cdot\Ev\left[(\nabla\Phi_M)^\intercal\nabla\Phi_M\right]m
                \right)\dx\\
    \Ze_\eff &= -\mu_0\int_\domain h_a\cdot \Ev\left[\Ms\right]m \dx.
\end{align*}
\end{thm}

\begin{remark}[Typical trajectories]\label{rem:typicaltrajectories}
The set $\wt{\Omega}\subset\Omega$ with $\P(\wt{\Omega})=1$, on which we prove the main theorem, is given by 
     $\wt{\Omega}=\wt{\Omega}_{\dP}\cap\wt{\Omega}_{\Ms\dP}\cap\wt{\Omega}_{\Ms^2\dP} \cap \wt{\Omega}_{a\dP}\cap\wt{\Omega}_{\kappa\dP}\cap\wt{\Omega}_{\kappa^2/a\dP}$ (see Theorem \ref{thm:lowerbound}). 
In other words, $\wt{\Omega}$ is the intersection of the sets of typical trajectories with respect to the measures $\dP$,  $a(\cdot)\dP(\cdot)$, $\kappa(\cdot)\dP(\cdot)$,   $\kappa^2/a(\cdot)\dP(\cdot)$,  $\Ms(\cdot)\dP(\cdot)$, and  $\Ms(\cdot)^2\dP(\cdot)$, respectively, see Definition \ref{def:typicaltrajectory}, Proposition \ref{prop:typicaltrajectories}, and Remark \ref{rem:mean-value property} below.
\end{remark}

\begin{remark}[Notation for the homogenization correctors]
\label{rk:correctors}
 We will address the existence and uniqueness of the corrector quantities given via \eqref{eq:mt:correctors} later in Subsection \ref{subs:correctors}.
 However, we phrase the stochastic framework in terms of an ergodic dynamical system.
 Thus, to adapt $\nabla\Phi_a,\nabla\Phi_\kappa,\nabla\Phi_M$ to that formulation, we introduce $\Theta_a,\Theta_\kappa,\Theta_M\in\Ltpot(\Omega;\Rt)\subset L^2(\Omega;\Rtt)$ instead, with \eqref{eq:mt:correctors} corresponding to 
 \begin{equation*}
     \left(a\Theta_a + a\Id\right), \, 
     \left(a\Theta_\kappa +\kappa\Id\right), \, 
     \left(\Theta_M + \Ms\Id\right) 
     \in\Ltsol(\Omega;\Rt),
 \end{equation*}
 see  Lemma \ref{lem:correctors} below for details.
\end{remark}
The key step for the proof of Theorem \ref{mainthm} is the analysis of the asymptotic behavior via $\Gamma$-convergence of the two exchange contributions. Since this latter result is of independent interest, we prove it in a broader setting. In what follows, $\M$ is a bounded, $C^2$ orientable hypersurface of $\Rt$ that admits a tubular neighborhood of uniform thickness (in the micromagnetic setting, usually $\M$ is the unit sphere $\St$). The tangent space at $s\in\M$ is denoted by $\T[s]\M$. The vector bundle $\boldT\M$ is given by $\boldT\M\coloneqq \bigcup_{s\in\M}\{s\}\times \bT[s]\M$, with $\bT[s]\M\coloneqq (\T[s]\M)^3$. We will indicate by $\xi^\intercal\coloneqq (\xi_1^\intercal, \xi_2^\intercal, \xi_3^\intercal)$ a generic element of $\bT[s]\M$. This notation is motivated by the fact that if $\xi\coloneqq\nabla m$, with $m\in H^1(D, \M)$ and $\nabla m$ is the transpose of the Jacobian matrix of $m$, the columns $(\xi_1^\intercal, \xi_2^\intercal, \xi_3^\intercal) \coloneqq (\partial_1 m(x), \partial_2 m(x), \partial_3 m(x))$ of $(\nabla m(x))^\intercal$ belong to $\T[m(x)]\M$ for a.e. $x\in D$.
Recalling \eqref{def:funcF} and \eqref{eq:intro:def:bchi}, our result is the following.
\begin{thm}[$\Gamma$-convergence of exchange energy and Dzyaloshinskii-Moriya interaction]\label{thm:GeGammaConv}
    Let $\M$ be a bounded, orientable, $C^2$ hypersurface of $\Rt$ admitting a tubular neighborhood of uniform thickness.
    Under the assumptions of Theorem \ref{mainthm}, for almost every $\tomega\in\Omega$ the family $(\G_\e)_{\e>0}\coloneqq (\E_\e+\K_\e)_{\e>0}$ $\Gamma$-converges with respect to the weak topology on $H^1(D;\hspace{0.03cm}\M)$ to the energy functional 
    \begin{align}
        \Ghom(m)
        &\coloneqq
        \frac{1}{2}\int_\domain \nabla m:\Ev\left[a\Id-(\nabla\Phi_a)^\intercal a\nabla\Phi_a\right]\nabla m\dx\notag\\
        &\quad
        -\int_\domain \nabla m :\Ev\left[\kappa\Id-(\nabla\Phi_a)^\intercal a\nabla\Phi_\kappa\right]\bchi(m)\dx\notag\\
        &\quad
        -\frac{1}{2}\int_\domain \bchi(m) :{\Ev}\left[(\nabla\Phi_\kappa)^\intercal a\nabla\Phi_\kappa\right]\bchi(m)\pi_{\T_s\M}\dx.\label{eq:intro:Ghom}
    \end{align}
    with $\nabla\Phi_a$, $\nabla\Phi_\kappa$ as in \eqref{eq:mt:correctors-intro} and where $\pi_{\T_s\M}\in\Rtt$ is the orthogonal projection onto the tangent space $\T_s\M$ at $s\in\M$.
    Note that $\bchi(m)\pi_{\T_s\M}=\bchi(m)$ for $\M=\St$.
\end{thm}
\begin{remark}[Anisotropic exchange parameter and DMI constant]
    We stress that Theorem \ref{thm:GeGammaConv} also holds (with the same proof) when the exchange parameter is given by a symmetric, uniformly elliptic coefficient field with $a(x,\omega)\in\Rtt$ and the DMI constant by a bounded coefficient field with $\kappa(x,\omega)\in\Rtt$ (interpreted as in \eqref{eq:intro:DMI}).
    However, to keep the notation simpler, we assume scalar parameters for the remainder of the paper (so that, e.g., we can write $a|\nabla m|^2$ instead of $\nabla m:(a\nabla m)$).
\end{remark}
The remaining energy components can be treated individually since they converge $\Gamma$-continuously. 

\begin{prop}[$\Gamma$-continuous convergence of the magnetostatic self-energy]\label{prop:intro:Whom}
  Under the assumptions of Theorem \ref{mainthm}, for almost every $\tomega\in\Omega$ the family $(\W_\e)_{\e>0}$ continuously $\Gamma$-converges with respect to the strong $L^2(\domain;\hspace{0.03cm}\St)$-topology and hence also with respect to the weak topology on $H^1(D;\hspace{0.03cm}\St)$ to the energy functional 
  \begin{equation}\label{eq:intro:def:Whom}
      \Whom(m)
      =
      -\frac{\mu_0}{2} \int_\domain h_d\left[\Ev\left[\Ms\right]m\right]\cdot \Ev\left[\Ms\right] m\dx
      + \frac{\mu_0}{2}\int_\domain m\cdot\Ev\left[(\nabla\Phi_M)^\intercal\nabla\Phi_M\right]m\dx
  \end{equation}
  with $\nabla\Phi_M$ as in \eqref{eq:mt:correctors}.
\end{prop}

\begin{prop}[$\Gamma$-continuous convergence of the anisotropy and Zeeman energy]\label{prop:intro:AhomZehom}
    Under the assumptions of Theorem \ref{mainthm}, for almost every $\tomega\in\Omega$ the families $(\A_\e)_{\e>0}$ and $(\Ze_\e)_{\e>0}$ continuously $\Gamma$-converge with respect to the strong $L^2(\domain;\hspace{0.03cm}\St)$-topology and hence also with respect to the weak topology on $H^1(D;\hspace{0.03cm}\St)$ to the energy functionals 
  \begin{align}
      \Ahom(m)
      &=
      \int_\domain\Ev\left[\vphi\right](m)\dx,\label{eq:intro:def:Ahom}\\
      \Zehom(m)
      &=
      -\mu_0\int_\domain h_a\cdot \Ev\left[\Ms\right]m \dx. \label{eq:intro:def:Zehom}
  \end{align}
\end{prop}
\BBB
\section{Stochastic two-scale convergence and homogenization correctors}
\label{section:sto2scalecon}
We recall here the basic properties of two-scale convergence and introduce the homogenization correctors.
For most of the proofs we rely on the Birkhoff Ergodic Theorem and on the concept of stochastic two-scale convergence (see, e.g., \cite{AdBMN21, ZP06} for a more extensive overview on this topic). 
For convenience of the reader, in this section we follow \cite[Section 2.2]{AdBMN21}, recalling the definitions and main results which we will apply in the forthcoming sections.
%In this section, let $(\Omega,\Sigma,\P)$ be a probability space with dynamical system $(T_x)_{x\in\Rt}$ as introduced in Section \ref{subs:assumptions} satisfying \ref{asT_SumGroupAction}--\ref{asT_compMetricSp}.
In what follows, recall that $D$ is a bounded, open domain in $\Rt$, and that $(\Omega,\Sigma,\P)$ is a probability space endowed with an ergodic dynamical system $(T_x)_{x\in \Rt}$.

\begin{thm}[Birkhoff Ergodic Theorem]
Let $f\in L^1(\Omega)$. 
Then, for $\P$-almost every $\tomega\in\Omega$ and for every bounded Borel set $A\subset \Rt$, there holds
    \begin{equation}
        \label{Birkhoff}
        \lim_{t\to\infty}{1\over t^3|A|}\int_{tA}f(T_x\tomega)\dx  = \int_\Omega f(\omega)\dPo =\Ev[f].
    \end{equation}
\end{thm}
\begin{remark}[Equivalent formulation of the Birkhoff Ergodic Theorem]
\label{rk:Birkhoff}
The previous result can be equivalently formulated in terms of weak$^\ast$ convergence of oscillating test functions. We refer the interested reader, e.g., to \cite[Theorem 2.9]{zeppieri} and \cite{krengel}. 
\end{remark}
Note that the set of $\tomega\in\Omega$ for which \eqref{Birkhoff} holds  depends on the function $f$. Hence, the Birkhoff Ergodic Theorem is not sufficient to obtain results valid almost surely for all functions $f$. 
This inconvenience can be avoided by restricting the study to the set of typical trajectories.

\begin{definition}
   \label{def:typicaltrajectory}
       Let $\widetilde{w}\in\Omega$. We say that $\tomega$ is a  {\it typical trajectory} if
          \begin{equation}
              \notag
              \lim_{t\to\infty}{1\over t^3|A|}\int_{tA} g(T_x\widetilde{\omega})\dx =\int_\Omega g(\omega)\dPo=\Ev[g] ,
          \end{equation}
    for every bounded Borelian $A\subset\R^3$ with $|A|>0$ and for every $g\in C(\Omega)$.
    The set of typical trajectories $\tomega$ will be denoted by $\widetilde{\Omega}_{\dP}$.
\end{definition}

Due to the separability of the space $C(\Omega)$ of continuous random variables, this restriction to the set of typical trajectories actually still includes almost every $\omega\in \Omega$.

\begin{prop}[{{\cite[Proposition 2.2.1]{AdBMN21}}}]
\label{prop:typicaltrajectories}
    Let $\widetilde{\Omega}_{\dP}$ be the set of typical trajectories. Then $\P(\widetilde{\Omega}_{\dP})=1$.
\end{prop}

For our purposes in this paper, rather than the original formulation of Birkhoff's Theorem, we will use the following extension on the set of typical trajectories obtained by rescaling and approximating compactly supported continuous via simple functions.

\begin{prop}[Mean-value property, {{\cite[Proposition 2.2.2]{AdBMN21}}}]
\label{prop:mean-value property}
Let $g\in C(\Omega)$. Then, for every $\varphi\in C_{\rm c}(\Rt)$ and for every $\widetilde{\omega}\in\widetilde{\Omega}_{\dP}$, there holds
    \begin{equation}
        \notag
        \lim_{\e\to 0} \int_{\Rt} \vphi(x)g(T_{x/\e}\widetilde{\omega})\dx  = \Ev[g]\int_{\R^3}\varphi(x)\dx .
    \end{equation}
\end{prop}

\begin{remark}[Typical trajectories and extension of the mean-value property]
\label{rem:mean-value property}
    While we will can often rely on the mean-value property as stated in Proposition \ref{prop:mean-value property}, we will mostly need a more general result due to the low regularity of the coefficients. 
    Let $\rho$ be a random variable, with $\rho\in L^1(\Omega)$.
    The measures $\dx$ and $\dPo$ in \eqref{Birkhoff} can be replaced with the {\it random stationary measure} $\m\mu_\omega=\rho(T_x\omega)\dx$ on $\Rt$  and the respective {\it Palm measure} $\m\f{\mu}(\omega)=\rho(\omega)\dPo$ on $\Omega$ (see \cite[Section 1 and Theorem 1.1]{ZP06} for the general definitions and the generalized ergodic theorem).
    Hence, $\tomega \in\Omega$ is  a $\rho\dP$-typical trajectory  if 
    \begin{equation*}
        \lim_{t\to\infty}\frac{1}{t^3|A|}\int_{tA}f(T_x\tomega)\m\mu_\tomega(x)
        = 
        \int_\Omega f(\omega)\m\f{\mu}(\omega),
    \end{equation*}
    for every bounded Borelian $A\subset\Rt$ with $|A|>0$ and for every $f\in C(\Omega)$. 
    We denote with $\wt{\Omega}_{\rho\dP}$ the set of $\rho\dP$-typical trajectories.
    Slightly adopting the arguments presented in \cite{AdBMN21}, we obtain that $\P(\wt{\Omega}_{\rho\dP})=1$ and an analogous version of the mean-value property holds in $\wt{\Omega}_{\rho\dP}$.
\end{remark}

We recall below the notion of stochastic two-scale convergence and its main properties.

\begin{definition}
\label{def:wstsconve}
Let $\widetilde{\omega}\in\wt{\Omega}_{\dP}$ be fixed, let $\{v_\e\}_{\e>0}$ be a family of elements of $L^2(D)$ and let $v\in L^2(D\times\Omega)$. We say that the subsequence $\{ v_{\e}\}$ weakly stochastically two-scale converges to $v$ if, for every $\psi\in C^\infty_{\rm c}(D)$ and every $b\in C(\Omega)$,
   \begin{equation}
   	\notag
   	\lim_{\e\to 0} \int_D v_{\e}(x)\psi(x)b(T_{x/\e}\widetilde{\omega})\dx =\int_D\int_\Omega v(x, \omega)\psi(x)b(\omega) \dPo\dx .
   \end{equation}
\end{definition}

We write 
    \begin{equation}
        \notag
        v_{\e}\in L^2(D)\wsts \,v\in L^2(D\times \Omega).
    \end{equation}
Note that the definition of two-scale convergence may depend on the choice of the typical trajectory $\widetilde{\omega}\in\wt{\Omega}_{\dP}$ and accordingly the limit may depend on $\widetilde{\omega}$. 
If not further specified, we refer to two-scale convergence with respect to the element $\tomega\in\widetilde{\Omega}$ in the definition of $(\F_\e)_{\e>0}$, see also Remark \ref{rem:typicaltrajectories}.\\
An important property of this convergence is the following compactness result.

\begin{thm} [{{\cite[Theorem 2.2.3]{AdBMN21}  }}]
\label{thm:L2cpt}
Let $\{ v_\e\}_{\e>0}$ be a bounded family in $L^2(D)$. 
Then, there exist a subsequence $\{\e_k\}_{k>0}$ going to zero and a function $v_0\in L^2(\Rt\times\Omega)$ such that $v_{\e_k}$ weakly two-scale-converges to $v_0$.
\end{thm}
We stress that, by Definition \ref{def:wstsconve}, any $\wt\omega$ might in principle give rise to a different two-scale limit. In other words, the specific $v_0$ identified in Theorem \ref{thm:L2cpt} will be dependent on the choice of $\wt\omega$ in Definition \ref{def:wstsconve}.

An explicit example for a class of weakly stochastically two-scale converging sequences is provided by a combination of a density argument and of the mean value property. 
%The Birkhoff theorem provides a large class of weakly two-scale-converging sequences, as the one below.

\begin{example}\label{ex:Birkhoff2scale}
Let $\vphi\in L^2(\Rt)=\overline{C_{\rm c}(\Rt)}^{\|\cdot\|_{L^2}}$, $g\in C(\Omega)$.
Then due to Proposition \ref{prop:mean-value property} the sequence $\lbrace v_\e\rbrace_{\e>0}$ given by 
    $v_\e(x)=\vphi(x)g(T_{x/\e}\tomega)$ 
weakly two-scale converges to 
    $v\in L^2(D\times\Omega)$ with $v(x,\omega)=\vphi(x)g(\omega)$.
\end{example}

\begin{comment}
Similarly we make the following observation, which we will use in the proof of $\Gamma$-continuous convergence for the magnetostatic self-energy.

\begin{example}\label{ex:Msatm2scale}
Let $m\in L^2(D,\M)\subset L^2(\Rt,\Rt)$ and $\tomega\in\wt{\Omega}=\wt{\Omega}_{\dP}\cap\wt{\Omega}_{\Ms\dP}$.
Then with Remark \ref{rem:mean-value property}, plugging in $\rho=\Ms$ from Assumption \ref{asPara_Msat}, the sequence $M_\e m$ given by $(M_\e m)(x)=\Ms(T_{x/\e}\tomega)m(x)$ weakly two-scale converges to $\Ms(\omega) m(x)$ component-wise.
\end{example}     
\end{comment}

Before providing the stochastic two-scale version of the weak-strong convergence principle, we introduce the notion of strong two-scale convergence.

\begin{definition}
A sequence $\{v_{\e}\}_{\e>0}\in L^2(D)$ is said to strongly two-scale converge to a function $v_0\in L^2(D\times\Omega)$ if $v_{\e}$ weakly two-scale converges to $v_0$ and if   
   \begin{equation}
   	\notag
   	\lim_{\e\to 0} \|v_{\e}\|_{L^2(D)} = \|v_0\|_{L^2(D\times\Omega)}.
   \end{equation}
\end{definition}
We write that $v_{\e}\in L^2(D) \sts\, v_0\in L^2(D\times\Omega).$
\begin{prop}[Weak-strong convergence principle, {{\cite[Proposition 2.2.4]{AdBMN21} }}]
\label{prop:weak-strong conv}
Let $\{ v_\e\}_{\e>0}$ and $\{ u_\e\}_{\e>0}$ be two families of functions in $L^2(D)$. 
If $\{ v_\e\}_{\e>0}$ strongly two-scale converges to $v_0$ and $\{ u_\e\}_{\e>0}$ weakly two-scale converges to $u_0$, for every $\psi\in C^\infty_{\rm c}(D)$ and $b\in C(\Omega)$,
  \begin{equation}
  \notag
  \lim_{\e\to 0} \int_D u_{\e}(x)v_{\e}(x)\psi(x)b(T_{x/\e}\widetilde{\omega})\dx  =\int_D\int_\Omega u_0(x, \omega)v_0(x, \omega)\psi(x)b(\omega)\dPo\dx .
  \end{equation}
\end{prop}
\begin{remark}[$L^2$-convergence vs. two-scale convergence]\label{rem:2sL2vsL2}
    As in the classical periodic case, on the one hand, strong $L^2$-convergence implies strong two-scale convergence. 
    Namely, if $u_\e \rightarrow u\in L^2(\Rt)$  strongly in $L^2(\Rt)$ as $\e\rightarrow 0$ for $\lbrace u_\e\rbrace\subset L^2(\Rt)$, then $u_\e \sts\, \wt{u}\in L^2(\Rt\times\Omega)$ with $\wt{u}(x,\omega)=u(x)$.\\
    On the other hand, weak two-scale convergence implies weak $L^2$-convergence.
    If $u_\e \wsts\, u\in L^2(\Rt\times\Omega)$, then $u_\e\rightharpoonup \Ev[u]$ weakly in $L^2(\Rt)$.
\end{remark}

\subsection{\texorpdfstring{\boldmath}{}The space\texorpdfstring{ ${\rm H^1(\Omega)}$}{H1}}
The goal of this subsection is to state a compactness result, similar to Theorem \ref{thm:L2cpt}, for sequences of uniformly bounded gradients.
%Rather than just in the space $L^2(\domain)$, we are interested in applying the compactness result Theorem \ref{thm:L2cpt} on the level of gradients, hence in $H^1(\domain)$.
The language to describe the emerging additional structure is provided via the spaces $H^1(\Omega)$, $\Ltpot(\Omega)$, and $\Ltsol(\Omega)$.
We briefly recall the definitions and main results.
%Here, we assume that $D$ is an open, bounded subset of $\Rt$.%\todos{J: see intro?}

First, we define a notion of derivatives on $\Omega$ that is compatible with the dynamical system $(T_x)_{x\in\Rt}$.

\begin{definition}
\label{def:derivates}
A function $u\in L^2(\Omega)$ is said to be differentiable at $\omega\in\Omega$ if  the limits
\begin{equation*}
    \lim_{\delta\rightarrow 0}\frac{u(T_{\delta e_i}\omega)-u(\omega)}{\delta}
    \eqqcolon (D_iu)(\omega)
\end{equation*}
exist for every $i\in\lbrace 1, 2, 3\rbrace$, where $(e_1,e_2, e_3)$ is the canonical basis of $\R^3$.
In this case, we write 
    $D_\omega u = (D_1 u,D_2 u, D_3 u)$,
    $\nablaomega u = (D_\omega u)^\intercal$,
and adapt the notation accordingly if $u$ is vector-valued.
\end{definition}

This allows us to define a space $C^1(\Omega)$ of continuously differentiable functions on the probability space $(\Omega, \Sigma, \P)$. 
Indeed, $C^1(\Omega)$ is the set of functions $u\in C(\Omega)$ that are differentiable, in the sense of Definition \ref{def:derivates}, at any $\omega\in\Omega$ and such that $\nablaomega u$ is continuous on $\Omega$. 
It turns out that $C^1(\Omega)$ is a dense subspace in $L^2(\Omega)$ (see \cite[Lemma 2.3.2]{AdBMN21}). 
Due to the ergodicity assumption, in particular via the mean-value property, we obtain an integration-by-parts result for these functions.

\begin{lemma} [{{\cite[Lemma 2.3.3]{AdBMN21}}}]
\label{lemma:stochintegrationparts}
Let $w\in C^1(\Omega)$. 
Then, for $i\in\{1, 2, 3\}$,  $\Ev[D_i w]=0$. 
In particular, for $u,v\in C^1(\Omega)$ there holds
   \begin{equation}
   \notag
   \int_\Omega (D_iu)v\dP+\int_\Omega u(D_iv)\dP=0.
   \end{equation} 
\end{lemma}

This leads to the definition of weak derivatives and thus to a Sobolev-type space $H^1(\Omega)$ on the probability space $(\Omega, \Sigma, \P)$.
%{\color{red}This space also is the closure of $C^1(\Omega)$ with respect to the canonical $H^1$-norm.}

\begin{definition}\label{def:weakderivative}
Let $u\in L^2(\Omega)$, $i\in\lbrace 1, 2, 3\rbrace$ and let $w_i\in L^2(\Omega)$. We say that $u$ admits $w_i$ as a (unique) weak derivative in the $i$-th direction  if 
\begin{equation*}
    -\int_\Omega u D_i v\dP = \int_\Omega w_iv\dP\,
    \quad\text{for all }v\in C^1(\Omega).
\end{equation*}
In this case, we denote by $D_iu=w_i$ the weak derivative of $u$ in the $i$-th direction.
\end{definition}
Note that both definitions of $D_i$ coincide for $C^1$-random variable. 
The space $H^1(\Omega)$ is the space of all functions $u\in L^2(\Omega)$ that admit weak derivatives in $L^2(\Omega)$ in all directions and we denote $D_\omega u= (D_1 u, D_2 u, D_3 u)$, $\nablaomega u = (D_\omega u)^\intercal$.
Clearly, $H^1(\Omega)$ is a vector space endowed with the inner product  
\begin{equation*}
    (u,v)_{H^1(\Omega)}\coloneqq (u,v)_{L^2(\Omega)} + (\nablaomega u, \nablaomega v)_{L^2(\Omega;\Rt)}.
\end{equation*}
\begin{prop} [{{\cite[Proposition 2.3.4]{AdBMN21}}}]
    The space $H^1(\Omega)$ is a separable Hilbert space and the subspace $C^1(\Omega)$ is dense in it. 
\end{prop}

%Hence $(H^1(\Omega), (\cdot, \cdot)_{H^1(\Omega)})$ is a separable Hilbert space and the subspace $C^1(\Omega)$ is dense. \\
\begin{remark}
    In \cite[Section 2]{ZP06}, the authors provide another definition of the Sobolev space $H^1(\Omega)$, which is similar to the one of  $H^1(\R, \mu)$  (see, e.g., \cite{BBS97, BF01, Z00}). For completeness, we recall the definition of $H^1(\Omega, \mu)$ below. Let $\mu$ be a Borel measure on $\Omega$. A function $u\in L^2(\Omega, \mu)$ belongs to $H^1(\Omega, \mu)$ and $z\in (L^2(\Omega, \mu))^3$ is the gradient of $u$ if there exists a sequence $\{u_k\}_{k>0}$ in $C^1(\Omega)$ such that $u_k\to u$ in $L^2(\Omega, \mu)$ and $D_i u_k\to z_i$ in $L^2(\Omega, \mu)$, for $i=1, 2, 3$, as $k\to \infty$.  The equivalence of the two definitions is shown in \cite[Section 4]{ZP06}.
    
\end{remark}

The ergodicity of the dynamical system implies that, if the weak derivatives of a Sobolev map vanish, then the original function is constant.

\begin{prop} [{{\cite[Proposition 2.3.6]{AdBMN21}}}]
\label{prop:uconstant}
For any $u\in H^1(\Omega)$, we have  
    \begin{equation}
    \notag
    \nablaomega u\equiv 0\quad\Longrightarrow\quad \mbox{u is constant $\P$-a.s.}
    \end{equation}
This property is also referred to as the measure $\P$ being ergodic with respect to translations.
\end{prop}

We proceed by introducing the spaces of gradients and divergence-free functions in $L^2(\Omega; \hspace{0.03cm}\Rt)$, which are used to define the homogenization correctors.  

\begin{definition}\label{def:LtpotLtsol}
We denote by $L^2_\pot(\Omega)$ the set of all functions of $L^2(\Omega;\Rt)$ which can be viewed as limits of a sequence of gradients. In other words,
  \begin{equation}
  \notag
  L^2_\pot(\Omega) \coloneqq
  \overline{\{\nablaomega u \hspace{0.03cm}:\hspace{0.03cm} u\in C^1(\Omega)) \}}^{\|\cdot\|_{L^2(\Omega;\hspace{0.03cm} \R^{3})}}.
  \end{equation}
We denote by $L^2_\sol(\Omega)$ the orthogonal complement of $L^2_\pot(\Omega)$.
\end{definition}

Finally, we state the $H^1$-extension of the $L^2$-compactness result from Theorem \ref{thm:L2cpt}. 

\begin{thm} [{{\cite[Theorem 2.4.2]{AdBMN21}}}]
\label{thm: compactnessH1}
Let $\{v_{\e}\}_{\e>0}$ be a bounded sequence in $H^1(D)$. Then, there exists $v_0$ in $H^1(D)$ and $\xi\in L^2(D, L^2_\pot (\Omega))$ such that, up to the extraction of a non-relabelled subsequence, there holds 
   \begin{align}
   	\notag
   	&v_\e \sts\, v_0(x) \quad\mbox{strongly in }L^2(D\times \Omega;\Rt),\notag\\
   	&\nabla v_\e \wsts\, \nabla v_0(x)+ \xi(x,\omega) \quad \mbox{ weakly in } L^2(D\times\Omega;\Rtt).\notag
   \end{align}
Moreover, if $v_\e\in H^1_0(\Omega)$ for every $\e>0$, then the limit $v_0$ is in $H^1_0(\Omega)$.
\end{thm}

\subsection{The correctors $\Theta_a$, $\Theta_\kappa$, and $\Theta_M$}\label{subs:correctors}
We generalize the spaces introduced in Definition \ref{def:LtpotLtsol} to vector-valued functions. 
The space $L^2_{\pot}(\Omega; \R^3)$ is defined as the closure of the set $\{ \nablaomega u\hspace{0.03cm} : \hspace{0.03cm} u\in C^1(\Omega; \Rt)\}$ in $L^2(\Omega; \R^{3\times 3})$.
Analogously to the scalar case, we set $L^2_{\sol}(\Omega; \Rt) \coloneqq (L^2_{\pot}(\Omega; \R^3))^{\perp}$ with respect to the $L^2(\Omega;\Rtt)$ scalar product 
    \begin{equation*}
        \left(A,B\right)_{L^2(\Omega;\Rtt)} 
        =
        \int_\Omega A(\omega):B(\omega)\dPo .
    \end{equation*} 
\begin{lemma}[Existence and uniqueness of the homogenization correctors]\label{lem:correctors}
    There exist unique solutions $\Theta_a,\Theta_\kappa,\Theta_M\in\Ltpot(\Omega;\Rt)$ to the problems
    \begin{align}
        \left(a\Theta_a + a\Id\right)           \in\Ltsol(\Omega;\Rt),
        \label{eq:correctors:Thetaa}\\
        \left(a\Theta_\kappa +\kappa\Id\right)  \in\Ltsol(\Omega;\Rt),
        \label{eq:correctors:Thetak}\\ 
        \left(\Theta_M + \Ms\Id\right)          \in\Ltsol(\Omega;\Rt).
        \label{eq:correctors:ThetaM}
    \end{align}
\end{lemma}

\begin{proof}
    As a closed subspace of $L^2(\Omega;\Rtt)$, the set $\Ltpot(\Omega;\Rt)$ is a Hilbert space. 
    Condition \eqref{eq:correctors:Thetaa} is equivalent to 
    \begin{equation*}
        \int_\Omega a(\omega) \Theta_a(\omega):B(\omega)\dPo 
        =
        -(a\Id,B)_{L^2(\Omega;\Rtt)}
        \quad\text{for all }B\in\Ltpot(\Omega;\Rt).
    \end{equation*}
    Hence, due to \ref{asPara_Ex} -- the uniform positivity and boundedness of $a$ -- there exists a unique solution due to the Lax--Milgram theorem.
    The same follows analogously for \eqref{eq:correctors:Thetak} and \eqref{eq:correctors:ThetaM} owing to \ref{asPara_DMI} and \ref{asPara_Msat}.
\end{proof}

\begin{remark}[The correctors via orthogonal projections]\label{rem:correctorsProjections}
 Note that $(\cdot,\cdot)_{a}\coloneqq (\cdot,a\,\cdot)_{L^2}$ is a scalar product on $L^2(\Omega;\Rt)$ or $L^2(\Omega;\Rtt)$, respectively, due to \ref{asPara_Ex}, and that such scalar product induces the same topology on the respective spaces.
 From this perspective, $\Theta_a e_i$ and $\Theta_\kappa e_i$ are the orthogonal projections with respect to $(\cdot,\cdot)_{a}$ of $\omega\mapsto -e_i$ and $
 \omega\mapsto -\frac{\kappa(\omega)}{a(\omega)}e_i$ onto $\Ltpot(\Omega)$ for $i=1,2,3$.
 Similarly, $\Theta_M$ is the orthogonal projection with respect to $(\cdot,\cdot)_{L^2}$ of $\omega\mapsto -\Ms e_i$ for $i=1,2,3$.
\end{remark}
\subsection{Two-scale Limits of fields in ${\rm H^1(D; \hspace{0.03cm}\M)}$} 
%\section{\texorpdfstring{\boldmath}{}Two-scale Limits of fields in\texorpdfstring{ ${\rm H^1(D; \hspace{0.03cm}\M)}$}{H1}} 
\label{section: twoscale lim}
We conclude this section with a  characterization of the stochastic two-scale limit of manifold-valued sequences in $H^1$. This result will be instrumental for the analysis in Section \ref{section:GammaresultExDMI}. Throughout this subsection, $\M$ will be a $C^2$ orientable hypersurface in $\R^3$. 
In this setting, two-scale limits will be the sum of standard $L^2$ gradients with elements of the subspace  $\Ltpot(\Omega; \T[s]\M)\subset \Ltpot(\Omega; \R^3)$ with $s\in\M$, which is defined as the closure of the set $\{ \nablaomega u\hspace{0.03cm} : \hspace{0.03cm} u\in C^1(\Omega, \T[s]\M)\}$ in $L^2(\Omega; \R^{3\times 3})$.

\begin{prop}
\label{prop:compactness}
Let $\M$ be a $C^2$ orientable hypersurface in $\R^3$ admitting a tubular neighborhood of uniform thickness, and let $\{u_\e\}_{\e>0}$ be a sequence of functions in $H^1(D; \hspace{0.03cm} \M)$. Let $u_0\in H^1(\domain; \hspace{0.03cm} \R^3)$ and $\Xi\in L^2(\domain; \hspace{0.03cm} \Ltpot(\Omega; \hspace{0.03cm}\R^3))$ be such that 
   \begin{align}
   	u_\e        &\to    u_0                             \qquad\mbox{strongly in } L^2(D; \hspace{0.03cm} \R^3), \label{cpt1}\\
   	\nabla u_\e &\wsts  \nabla u_0(x)+\Xi(x, \omega)    \qquad \mbox{weakly in } L^2(D\times\Omega; \hspace{0.03cm} \R^{3\times 3}).\label{cpt2}
   \end{align}
Then, $u_0\in H^1(\domain; \hspace{0.03cm} \M)$ and  $\Xi\in L^2(\domain; \hspace{0.03cm} L^2_\pot(\Omega; \hspace{0.03cm} \T[u_0(x)]\M))$.
\end{prop}
\begin{proof}
The assumptions on $\M$ ensure the existence of an open tubular neighborhood $U\in\R^3$ of $\M$ and a $C^2$ function $\gamma: U\to\R$ which has $0$ as a regular value and with $\M= \gamma^{-1}(\lbrace 0\rbrace)$. 
Thanks to \eqref{cpt1} combined with the continuity of $\gamma$ we deduce that, up to the extraction of a non-relabelled subsequence, $0=\gamma (u_\e(x))\to \gamma(u_0(x))=0$ for almost every $x\in D$. 
This implies that $u_0(x)\in\M$ for almost every $x\in D$, and hence $u_0\in H^1(\domain,\M)$.  
\par 
Next, we show that $\Xi(x, \cdot)\in\Ltpot(\Omega; \hspace{0.03cm} \T[u_0(x)]\M)$ for almost every $x\in\domain$.
For $x\in D$ fixed, we need to prove that there exists a sequence $\lbrace \wt{v}_n\rbrace_{n\in\N}\subset C^1(\Omega,\T[u_0(x)]\M)$, possibly dependent on $x$, and such that 
   \begin{equation}
       \label{appr1}
       \|\nablaomega \wt{v}_n-\Xi(x,\cdot)\|_{L^2(\Omega; \hspace{0.03cm} \R^{3\times 3})}\to 0 \qquad \mbox{as } n\to\infty.
   \end{equation}
First, note that since $\Xi(x,\cdot)\in \Ltpot(\Omega; \hspace{0.03cm}\R^3)$, there exists a sequence $\lbrace v_n\rbrace_{n\in\N}\subset C^1(\Omega; \hspace{0.03cm} \R^3)$ such that 
   \begin{equation}
       \label{appr2}
    \|\nablaomega v_n-\Xi(x,\cdot)\|_{L^2(\Omega; \hspace{0.03cm} \R^{3\times 3})}\to 0 \qquad \mbox{as } n\to\infty.
   \end{equation}
We define the maps $\widetilde{v}_n\in C^1(\Omega,\T[u_0(x)]\M)$ as the nearest point projections of $v_n$ onto $\T[u_0(x)]\M$. Namely, 
\begin{equation}
    \notag
    \widetilde{v}_n(\omega)
    \coloneqq
    v_n(\omega) - \left(v_n(\omega) \cdot \frac{\nabla \gamma(u_0(x))}{ |\nabla \gamma(u_0(x))|} \right){\nabla \gamma(u_0(x))\over |\nabla \gamma(u_0(x))|}.
\end{equation}
Now, 
   \begin{align}
       \|\nablaomega \wt{v}_n-\Xi(x,\cdot)\|_{L^2(\Omega; \hspace{0.03cm} \R^{3\times 3})} 
       &\leq 
       \|\nablaomega \wt{v}_n-\nablaomega v_n\|_{L^2(\Omega; \hspace{0.03cm} \R^{3\times 3})}
       +\|\nablaomega v_n-\Xi(x,\cdot)\|_{L^2(\Omega; \hspace{0.03cm} \R^{3\times 3})}. \notag
   \end{align}
To prove \eqref{appr1}, thanks to \eqref{appr2}, it is enough to show that $\|\nablaomega \wt{v}_n-\nablaomega v_n\|_{L^2(\Omega; \hspace{0.03cm} \R^{3\times 3})} \to 0$ as $n\to\infty$. To this end, note that for any $\e>0$, there holds
    \begin{equation}
    \notag
    0=\nabla (\gamma\circ u_\e) = \nabla u_\e \nabla \gamma(u_\e).
    \end{equation}
From \eqref{cpt1}, it follows that $\nabla \gamma(u_\e) \to\nabla \gamma(u_0)$ strongly in $L^2(D; \hspace{0.03cm} \R^3)$. Thus, in view of \eqref{cpt2} we obtain that, for every $\varphi\in C^\infty_{\rm c} (D)$ and for every $b\in C(\Omega; \hspace{0.03cm} \R^3)$, 
\begin{multline*}
   	0 =
   	\int_D \nabla u_\e(x) \nabla\gamma (u_\e)\cdot b(T_{x/\e} \widetilde{\omega})\varphi(x) \dx 
    \to \int_D\int_\Omega (\nabla u_0(x)+\Xi(x, \omega))\nabla\gamma(u_0)\cdot b(\omega)\varphi(x)\dPo\dx
\end{multline*}
as $\e\to 0$. Since $\nabla \gamma(u_0)\nabla u_0=\nabla(\gamma\circ u_0)=0$, we conclude that for every $\varphi\in C^\infty_{\rm c} (D)$ and for every $b\in C(\Omega; \hspace{0.03cm} \R^3)$ 
   \begin{equation}
   	\notag
   	\int_D\int_\Omega \Xi(x, \omega) \nabla\gamma(u_0)\cdot b(\omega)\varphi(x)\dPo\dx  =0. 
   \end{equation}
It follows that, almost everywhere in $D$, and for every $b\in C(\Omega; \hspace{0.03cm} \R^3)$, 
    \begin{equation}
    \label{eq:valueontangent1}
        \int_\Omega \Xi(x, \omega) \nabla\gamma(u_0)\cdot b(\omega)\dPo =0.
    \end{equation}
Therefore, for almost every $x\in\domain$, we deduce that 
\begin{align*}
    &\|\nablaomega(v_n - \wt{v}_n)\|_{L^2(\Omega; \R^{3\times 3})}^2
    =
    |\nabla \gamma(u_0(x))|^{-4}
    \|\nabla \gamma(u_0(x)) \otimes \nablaomega v_n\nabla \gamma(u_0(x))\|_{L^2(\Omega; \R^{3\times 3})}^2\\
    &\qquad\leq 
    2|\nabla \gamma(u_0(x))|^{-2}\Big(
    \|(\nablaomega v_n-\Xi(x,\cdot))\nabla \gamma(u_0(x))\|_{L^2(\Omega; \Rt)}^2+\|\Xi(x,\cdot)\nabla \gamma(u_0(x))\|_{L^2(\Omega; \Rt)}^2\Big)\\
    &\qquad\leq
    2\|\nablaomega v_n-\Xi(x,\cdot)\|_{L^2(\Omega; \R^{3\times 3})}^2
\end{align*}
where in the last inequality we have used \eqref{eq:valueontangent1} by approximating $\Xi(x,\cdot)\nabla \gamma(u_0(x))$ in $C(\Omega,\Rt)$. 
This along with \eqref{appr2} implies \eqref{appr1} and concludes the proof.
\end{proof}

\section{The limit behavior of the exchange and DMI energies}\label{section:GammaresultExDMI}
In this section we prove Theorem \ref{thm:GeGammaConv}. In what follows, $\M$ is a bounded, $C^2$ orientable hypersurface of $\Rt$ that admits a tubular neighborhood of uniform thickness. 
The tangent space at $s\in\M$ is denoted by $\T[s]\M$. 
The vector bundle $\boldT\M$ is given by $\boldT\M\coloneqq \bigcup_{s\in\M}\{s\}\times \bT[s]\M$, with $\bT[s]\M\coloneqq (\T[s]\M)^3$. We will indicate by $\xi^\intercal\coloneqq (\xi_1^\intercal, \xi_2^\intercal, \xi_3^\intercal)$ a generic element of $\bT[s]\M$.  
As a corollary of Theorem \ref{thm:GeGammaConv}, in the special case $\M=\St$ we will infer a characterization of
the asymptotic behavior of the sum of exchange energy and DMI term. Recall that 
\begin{align}
    \G_\e(m)
    &\coloneqq \E_\e(m)+\K_\e(m) \notag\\
    & ={1\over 2} \int_\domain a(T_{x/ \e}\tomega)\nabla m(x):\nabla m(x)\dx  
                - \int_\domain \kappa(T_{x/\e}\tomega) \bchi(m(x)):\nabla m(x)\dx \notag\\
    & ={1\over 2} \int_\domain a(T_{x/ \e}\tomega)\left|\nabla m(x)-\frac{\kappa(T_{x/\e}\tomega)}{a(T_{x/\e}\tomega)} \bchi(m(x))\right|^2\dx  
                - \int_\domain \frac{\kappa^2(T_{x/\e}\tomega)}{a(T_{x/\e}\tomega)} |m(x)|^2\dx,\label{def:funncEx+DMI}
\end{align}
for every $m\in H^1(D; \M)$, where $\bchi$ is the map defined in \eqref{eq:intro:def:bchi} and $\tomega\in\wt{\Omega}$ as in Remark \ref{rem:typicaltrajectories}. 

This section is organized as follows. First, in Subsection \ref{subsect:homDensity} we provide an alternative formulation of the homogenized energy $\Ghom$ defined in \eqref{eq:intro:Ghom} via a pointwise minimization problem on $\Ltpot(\Omega;\hspace{0.03cm} \T[s]\M)$.
In Subsection \ref{subsect:liminfine}, we prove a compactness result as well as the liminf inequality. Finally, Subsection \ref{subsect:limsupine} is devoted to the proof of the optimality of the lower bound identified in Subsection \ref{subsect:liminfine}.

\subsection{Description of the tangentially homogenized energy density}\label{subsect:homDensity}

\begin{prop}[Characterization of the homogenized energy density]\label{prop:ThomGhom}
    Under the assumptions of Theorem \ref{thm:GeGammaConv}, there holds
    \begin{equation}
    \notag
       \Ghom(m) =  \int_D T_\hom(m(x), \nabla m(x))\dx,
    \end{equation}
    for every $m\in H^1(\domain;\M)$, where, for every $s\in \M$ and $A^\intercal\in \bT[s]\M\subset\Rtt$, the tangentially homogenized energy density is defined as
    \begin{multline}\label{eq:def:Thom}
        T_\hom(s, A) 
        \coloneqq 
        \min\biggl\lbrace
            \int_\Omega \left(\frac{1}{2}a(\omega)|A+\Xi(\omega)|^2 - \kappa(\omega) \bchi(s):(A+\Xi(\omega)) \right)\dPo :
            \Xi\in \Ltpot(\Omega;\hspace{0.03cm} \T[s]\M)
        \biggr\rbrace.
    \end{multline}
    Additionally, the minimum problem above admits a unique solution, given by 
    \begin{equation}\label{eq:ThomMinimzer}
        \Xi[s,A] =  \Theta_a A - \Theta_\kappa \bchi(s)\pi_{\T[s]\M}.
    \end{equation}
\end{prop}

\begin{proof}
 We will prove the proposition in 2 steps.
 First, we will show the characterization \eqref{eq:ThomMinimzer} of the unique solution to the minimization problem \eqref{eq:def:Thom}.
 Second, plugging this minimizer back into \eqref{eq:def:Thom}, we will obtain the density of $\Ghom$ as defined in \eqref{eq:intro:Ghom}.

 \underline{Step 1: Finding the minimizer for \eqref{eq:def:Thom}.} 
 We set 
 \begin{align*}
     \Xi[s,A]
     &\coloneqq
     \argmin_{\Xi\in\Ltpot(\Omega;\T[s]\M)} \left\lbrace\int_\Omega \left({1\over 2}a(\omega)|A+\Xi(\omega)|^2 - \kappa(\omega) \bchi(s):(A+\Xi(\omega)) \right)\dPo\right\rbrace
 \intertext{and via the quadratic structure, omitting the constant terms $-\frac{1}{2}\int_\Omega \frac{\kappa^2(\omega)}{a(\omega)}|\bchi(s)|^2\dPo$, we obtain }
     &=
     \argmin_{\Xi\in\Ltpot(\Omega;\T[s]\M)} \left\lbrace  {1\over 2}\int_\Omega a(\omega)\left|A+\Xi(\omega)-\frac{\kappa(\omega)}{a(\omega)} \bchi(s)\right|^2 \dPo\right\rbrace\\
     &=
     \argmin_{\Xi\in\Ltpot(\Omega;\T[s]\M)} \left\lbrace \left(\Xi+A-\kappa a^{-1} \bchi(s),\,\Xi+A-\kappa a^{-1} \bchi(s) \right)_a\right\rbrace
 \end{align*}
 where $(\cdot,\cdot)_a$ is the scalar product on $L^2(\Omega;\Rtt)$ introduced in Remark \ref{rem:correctorsProjections}.
 Hence, $\Xi[s,A]$ is given as the orthogonal projection with respect to $(\cdot,\cdot)_a$ of $-A+\kappa a^{-1} \bchi(s)$ onto $\Ltpot(\Omega;\hspace{0.03cm}\T[s]\M)$.
 On the one hand, due to Remark \ref{rem:correctorsProjections}, the projection onto $\Ltpot(\Omega;\hspace{0.03cm}\Rt)$ is given by $\Theta_a A - \Theta_\kappa \bchi(s)$.
 On the other hand, the projection from $\Ltpot(\Omega;\hspace{0.03cm}\Rt)$ onto $\Ltpot(\Omega;\hspace{0.03cm}\T[s]\M)$ is given by right-multiplication with $\pi_{\T[s]\M}\in\Rtt$ (the representation of the projection from $\Rt$ onto $\T[s]\M$).
 Hence, 
 \begin{equation}\label{eq:homG:minimizer}
     \Xi[s,A] = \Theta_a A \pi_{\T[s]\M} - \Theta_\kappa \bchi(s)\pi_{\T[s]\M}
    = \Theta_a A - \Theta_\kappa \bchi(s)\pi_{\T[s]\M},
 \end{equation}
 where we used that $A \pi_{\T[s]\M} = (\pi_{\T[s]\M} A^\intercal)^\intercal = A$ with the first equation following from $\pi_{\T[s]\M}=(\pi_{\T[s]\M})^\intercal$ and the second from $A^\intercal\in \bT[s]\M$.

 \underline{Step 2: Plugging $\Xi[s,A]$ back into \eqref{eq:def:Thom}.} 
 Due to the properties of orthogonal projections, we have that for all $\Psi\in\Ltpot(\Omega;\hspace{0.03cm}\T[s]\M)$
 \begin{equation*}
     \left(\Xi[s,A]+A-\kappa a^{-1} \bchi(s),\, \Psi\right)_a = 0.
 \end{equation*}
 Since in particular $\Xi[s,A]\in\Ltpot(\Omega;\hspace{0.03cm}\T[s]\M)$, it follows that
 \begin{equation*}
     \int_\Omega \kappa \bchi(s):\Xi[s,A] \dP = \int_\Omega a \left(\Xi[s,A]+A\right):\Xi[s,A] \dP. 
 \end{equation*}
 Plugging this into \eqref{eq:def:Thom}, resolving the square, and then applying  \eqref{eq:homG:minimizer}, we obtain
 \begin{align}
     &\nonumber T_\hom(s, A) 
     = 
     \int_\Omega \left( \frac{a}{2}|A|^2 - \frac{a}{2}|\Xi[s,A]|^2 -\kappa \bchi(s):A \right) \dP\\
     &\label{eq:explicit}\quad=
     \int_\Omega \bigg( \frac{a}{2}A:A 
                        - \frac{a}{2}A:\Theta_a^\intercal\Theta_a A
                        + aA:\Theta_a^\intercal\Theta_\kappa\bchi(s) 
                        - \frac{a}{2}\bchi(s):\Theta_\kappa^\intercal\Theta_\kappa\bchi(s)\pi_{\T[s]\M}
                        - \kappa A:\bchi(s)
                    \bigg)\dP
 \end{align}
 where for the last step we additionally used that $X:(ZY^\intercal)=(XY):Z=Y:(X^\intercal Z)$ for any $X,Y,Z\in\Rtt$, which also yields the absorption of $\pi_{\T[s]\M}$ in the mixed term as at the end of Step 1.
 Also note that $\pi_{\T[s]\M}(\pi_{\T[s]\M})^\intercal = (\pi_{\T[s]\M})^2 =\pi_{\T[s]\M}$.
 Reordering the energy terms yields the density in \eqref{eq:intro:Ghom} and thus completes the proof. Note that the measurability of $x\to T_\hom(m(x), \nabla m(x))$ follows directly by the explicit expression in \eqref{eq:explicit}.
\end{proof}

\begin{remark}[The micromagnetic case]
    For $\M=\St$ we observe $\bchi(s)\pi_{\T[s]\M}=\bchi(s)$ and -- following the proof above -- could therefore omit the projection from $\Ltpot(\Omega;\hspace{0.03cm}\Rt)$ onto $\Ltpot(\Omega;\hspace{0.03cm}\T[s]\M)$.
    This means, that the minimizer of \eqref{eq:def:Thom} within $\Ltpot(\Omega;\hspace{0.03cm}\Rt)$ is automatically in $\Ltpot(\Omega;\hspace{0.03cm}\T[s]\M)$:
    \begin{equation*}
        T_\hom(s, A) 
        =
        \min_{\Xi\in\Ltpot(\Omega;\Rt)} \left\lbrace\int_\Omega \left({1\over 2}a(\omega)|A+\Xi(\omega)|^2 - \kappa(\omega) \bchi(s):(A+\Xi(\omega)) \right)\dPo\right\rbrace.
    \end{equation*}
    On a geometric level, this is the case because $s$ is the normal vector of $\T[s]\St$.
    An alternative generalization from $\St$ to arbitrary manifolds is given by replacing \eqref{eq:intro:def:bchi} with 
    \begin{equation*}
        \bchi\colon s\in\M\mapsto (e_1\times \nu_s, e_2\times \nu_s, e_3\times \nu_s)^\intercal\in \R^{3\times 3}_{\rm skew-symmetric},
    \end{equation*}
    where $\nu_s$ is the normal vector of $\T[s]\M$. This choice would allow to relax to the minimum over all $\Xi\in\Ltpot(\Omega;\Rt)$ also in the case of general manifolds, making the additional projection in \eqref{eq:intro:Ghom} unnecessary.
\end{remark}

%\todos{J: Do we still neeed to mention measurability somewhere?}
%The proof of Theorem \ref{thm:Gammaconv ex+DMI} consists in proving a lower and upper bound for the family $\G_\e$. This will be the subject of the forthcoming subsections.

\subsection{The liminf inequality}
\label{subsect:liminfine}
In this subsection, we provide a compactness result, as well as a lower bound for the asymptotic behavior of the energy $\G_\e$ using the characterization of $\Ghom$ from Proposition \ref{prop:ThomGhom}. 
%We will denote by $L^2(D,\hspace{0.03cm} \Ltpot(\Omega, \T[m_0]\M) )$ the space of matrix-valued functions $\Xi\in\Ltpot(\Omega, \Rt))$ such that $\Xi(x, \omega)\in \bT_{m_0}\M$ for a.e. $x\in D$ and for a.s. $\omega\in\Omega$.

%\todo{Note: The actual definition is that $\Ltpot(\Omega,\T_{m_0}\M)$ is the closure of the gradients of $C^1$-functions from $\Omega$ to $\T[m_0]\M$}

\begin{thm}
\label{thm:lowerbound}
Let $\{m_\e\}_{\e>0}\subset H^1(D; \hspace{0.03cm}\M)$ be such that $\sup_{\e> 0}\G_\e(m_\e)<\infty$,  and let $\wt{\omega}\in\wt{\Omega}$, cf. Remark \ref{rem:typicaltrajectories}. 
Then, there exists $m_0\in H^1(D; \hspace{0.03cm} \M)$ and $M_1\in L^2(D; \Ltpot(\Omega; \T[m_0]\M))$ such that 
   \begin{align}
       m_\e&\rightharpoonup m_0 \quad \mbox{weakly in } H^1(\Omega; \hspace{0.03cm} \M)\notag\\
       \nabla m_\e &\wsts\nabla m_0(x) +M_1(x, \omega) \quad \mbox{weakly in } L^2(D\times\Omega; \hspace{0.03cm} \Rtt).\notag
   \end{align}
Moreover, 
    \begin{align}
        \label{liminfineq}
        \Ghom(m_0) \leq \liminf_{\e\to 0} \G_\e(m_\e),
    \end{align}
where $\Ghom$ is given by \eqref{eq:intro:Ghom}.
\end{thm}
\begin{proof}
The compactness result is a direct consequence of Proposition \ref{prop:compactness}. To prove \eqref{liminfineq}, we write $\G_\e(m_\e)$ as the sum of two functionals ${\mathcal{H}}_\e$ and ${\mathcal{I}}_\e$ given by 
   \begin{align}
       \notag
       {\mathcal{H}}_\e(m_\e):= {1\over 2} \int_D a(T_{x/\e}\widetilde{\omega})\left|\nabla m_\e(x)-{\kappa(T_{x/\e}\widetilde{\omega})\over a(T_{x/\e}\widetilde{\omega})}\bchi(m_\e(x))\right|^2\dx ,
   \end{align}
and
   \begin{align}
       \notag
       {\mathcal{I}}_\e(m_\e):= - \int_D {\kappa^2(T_{x/\e}\widetilde{\omega})\over a(T_{x/\e}\widetilde{\omega})} |m_\e(x)|^2\dx .
   \end{align}
    Let $\{\psi_{0,n}\}_{n\in\N}\subset C^\infty(D; \R^3)$ and let $\{\psi_{1, n}\}_{n\in\N}$ be a sequence of elements in the linear span of admissible test functions in $C_c^{\infty}(D) \times C(\Omega; \hspace{0.03cm} \R^{3\times 3})$  for stochastic two-scale convergence.  We define 
        \begin{equation}
            \notag
            \Psi_n(x, \omega) \coloneqq \nabla \psi_{0,n}(x) + \psi_{1, n}(x, \omega) - {\kappa(\omega)\over a(\omega)}\bchi(\psi_{0, n}(x)).
        \end{equation}
    For every $\e>0$ and by the positivity of $a$,  we find 
      \begin{align}
      \notag
          {1\over 2}\int_D a(T_{x/\e}\tomega)\left| \nabla m_\e(x) -{\kappa(T_{x/\e}\tomega)\over a(T_{x/\e})\tomega} \chi(m_\e(x))  - \Psi_n(x, T_{x/\e}\tomega) \right|^2\dx>0.
      \end{align}
    Hence, 
    \begin{align}
           {\mathcal{H}}_\e(m_\e) 
           &\geq -{1\over 2} \int_D a(T_{x/\e}\widetilde{\omega})|\Psi_n(x,T_{x/\e}\widetilde{\omega})|^2\dx  \notag\\
           &\quad +\int_D \left(a(T_{x/\e}\widetilde{\omega})\nabla m_\e(x) - \kappa(T_{x/\e}\widetilde{\omega})\bchi (m_\e(x))\right) : \Psi_n(x,T_{x/\e}\widetilde{\omega})\dx .  \notag
       \end{align}  
    Since $\psi_{1, n}$ belongs to the linear span of admissible test functions and 
    since $\tomega$ is a typical trajectory with respect to $a\dP,\kappa\dP,\kappa^2/a\dP$, see Remark \ref{rem:typicaltrajectories},
    we can pass to the stochastic two-scale convergence, obtaining
   \begin{align}
       &\liminf_{\e\to 0} {\mathcal{H}}_\e(m_\e) \geq -{1\over 2} \int_D\int_{\Omega} a(\omega)|\Psi_n(x,\omega)|^2\dPo\dx \notag\\
       &\quad +\int_D\int_{\Omega} a(\omega) \left[\nabla m_0(x) + M_1(x, \omega)-{\kappa(\omega)\over a(\omega)} \bchi(m_0(x)) \right] : \Psi_n(x, \omega)\dPo\dx .\label{ineliminf1}
   \end{align}
Choosing $\psi_{0, n}$ and $\psi_{1,n}$ such that $\psi_{0, n}\to m_0$ in $H^1(D; \hspace{0.03cm} \R^3)$ and $\psi_{1, n}\to M_1$ in $L^2(D\times \Omega; \hspace{0.03cm} \R^{3\times 3})$, and taking the limit as $n\to\infty$ in \eqref{ineliminf1}, we deduce that 
   \begin{align}
        &\liminf_{\e\to 0} {\mathcal{H}}_{\e} (m_\e)  
        \geq
        {1\over 2}\int_D\int_\Omega a(\omega) 
            \left|\nabla m_0(x) + M_1(x,\omega) - {\kappa(\omega)\over a(\omega)}\bchi(m_0(x))\right|^2
        \dPo \dx\notag\\
        &\quad\geq 
        {1\over 2}\int_D \inf_{\Xi\in L^2(\Omega; \hspace{0.03cm} \T[m_0(x)]\M)} 
            \int_\Omega a(\omega) \left|\nabla m_0(x)+\Xi(\omega)
                -{\kappa(\omega)\over a(\omega)}\bchi(m_0(x))\right|^2
            \dPo \dx\notag\\
        &\quad\geq  
        \int_D \inf_{\Xi\in L^2(\Omega; \hspace{0.03cm} \T[m_0(x)]\M)} 
            \int_{\Omega} {1\over 2}a(\omega)|\nabla m_0(x) + \Xi(\omega)|^2 
                -\kappa(\omega) \bchi(m_0(x)):(\nabla m_0(x) + \Xi(\omega))
            \dPo \dx\notag\\
        &\qquad +{1\over 2}\Ev\left[{\kappa^2\over a}\right]\int_D |m_0(x)|^2\dPo \dx . \label{ineliminf2}      
   \end{align}
 Concerning ${\mathcal{I}}_\e(m_\e)$ we have 
\begin{align}
        \lim_{\e\to 0} {\mathcal{I}}_\e(m_\e) 
        = -{1\over 2}\Ev\left[{\kappa^2\over a}\right]\int_D|m_0(x)|^2\dx
        \eqqcolon {\mathcal{I}}_0(m_0),\label{ineliminf3}
\end{align}
since 
\begin{align*}
    |{\mathcal{I}}_\e(m_\e)- {\mathcal{I}}_0(m_0)|\leq |{\mathcal{I}}_\e(m_\e)-{\mathcal{I}}_\e(m_0)| + |{\mathcal{I}}_\e(m_0)-{\mathcal{I}}_0(m_0)| \to 0
\end{align*}
as $\e\to 0$ and $m_\e\to m_0$ strongly in $L^2(D; \hspace{0.03cm}\M)$.
The convergence of the first term is due to the bounds on $a,\kappa$ from \ref{asPara_Ex} and \ref{asPara_DMI}. 
The convergence of the second term holds due to the boundedness of $\M$ and since $\tomega$ is a typical trajectory with respect to $\kappa^2/a(\cdot)\dP(\cdot)$, see Remark \ref{rem:mean-value property}.
Combining \eqref{ineliminf2} and \eqref{ineliminf3}, via Proposition \ref{prop:ThomGhom} we obtain \eqref{liminfineq}.
\end{proof}

\subsection{The limsup inequality}
\label{subsect:limsupine}
In this subsection, we provide an upper bound for the asymptotic behavior of the family $(\G_\e)_{\e>0}$. 
To be precise, we show that the lower bound identified in Theorem \ref{thm:lowerbound} is optimal. 
\begin{thm}
Let $\M$ be a $C^2$ orientable hypersurface of $\Rt$ such that $\M$ has a tubular neighborhood of uniform thickness $\delta>0$.
Let $m_0\in H^1(\domain;\hspace{0.03cm}\M)$.
Then, there exists a sequence $\{m_\e\}_{\e>0}\subset H^1(\domain;\hspace{0.03cm}\M)$, such that, as $\e\to 0$,
\begin{align*}
    m_\e\rightharpoonup m_0
    \quad \text{weakly in }H^1(\domain;\hspace{0.03cm}\M),
\end{align*}
and
\begin{align*}
    \Ghom(m_0)\geq\limsup_{\e\rightarrow 0} \G_\e(m_\e),
\end{align*}
where the functional $\Ghom$ is given by \eqref{eq:intro:Ghom}.
\end{thm}

\begin{proof}
 Let $U_\delta$ be a tubular neighborhood of size $\delta$ around $\M$.
 Denote by $\pi_\M\colon U_\delta\rightarrow\M$ the pointwise projection onto the manifold. 
 Since $\M$ is $C^2$, we have $\pi_\M\in C^1(U_\delta; \hspace{0.03cm}\M)$. In particular, we choose $\delta>0$ small enough such that  $\pi_\M\in C^1(\ov{U}_\delta; \hspace{0.03cm}\M)$.
 We split the proof into two steps.\\
 
 \underline{Step 1: Recovery sequence for smooth 2-scale  perturbations of $m_0$.}
 %Fixed the prototypical perturbation of some $m_0\in H^1(\domain;\hspace{0.03cm}\M)$. 
 For  $i=1,\ldots,n$, with $n\in\N$,  consider $\vphi_i\in C^\infty_{c}(\domain;\hspace{0.03cm}\Rtt)$, $b_i\in C^1(\Omega;\hspace{0.03cm}\Rt)$ for $i=1,\ldots,n$. 
  For every $\e>0$ and for almost every $x\in D$, we set
 \begin{align}%\label{eq:limsup:L2conv}
     \widetilde{m}_\e(x)   &\coloneqq m_0(x)+\e\sum_{i=1}^n\vphi_i(x) b_i(T_{x/\e}\widetilde{\omega}),\notag\\
     m_\e           &\coloneqq \pi_\M\circ\wt{m}_\e\in H^1(\domain; \hspace{0.03cm}\M).\notag
 \end{align}
For $\e>0$ sufficiently small , we find $\wt{m}_\e(x)\in U_\delta$ for almost every $x\in\domain$.
 Additionally, as $\e\rightarrow 0$,
 \begin{align}
 \label{strongconvergein L^2}
     \wt{m}_\e,m_\e\rightarrow m_0   \quad\text{strongly in }L^2(\domain; \hspace{0.03cm}\Rt).
 \end{align}
Thanks to the regularity of $\pi_\M$, there exists a positive constant $C_\M$ depending only on $\M$ such that 
 \begin{align*}
   |\nabla m_\e|\leq  C_\M|\nabla \wt{m}_\e|    \quad\text{a.e. in }\domain.
 \end{align*}
To be precise, a direct computation shows that 
 \begin{align}
%\label{eq:limsup:gradme}
     (\nabla \wt{m}_\e(x))_{jk} 
     &= 
     (\nabla m_0(x))_{jk} 
     + \e\sum_{i=1}^n\sum_{\ell=1}^3 b_{i,\ell}^\intercal(T_{x/\e}\wt{\omega})\partial_{x_j}\vphi_{i,k\ell}(x)
     + \sum_{i=1}^n (\nablaomega b_i(T_{x/\e}\wt{\omega})\vphi_i^\intercal(x))_{jk},\notag\\
     \nabla m_\e(x) 
     &= 
     \nabla \wt{m}_\e(x)\nabla \pi_\M[\wt{m}_\e(x)].\notag
 \end{align}
Since $\{\wt{m}_\e\}_\e$ is bounded in $H^1(\domain; \hspace{0.03cm}\Rtt)$, the sequence $\{m_\e\}_\e$ is bounded in $H^1(\domain; \hspace{0.03cm}\Rtt)$ as well. This, combined with \eqref{strongconvergein L^2} implies that, up to the extraction of a non-relabelled subsequence, 
 \begin{align*}
     \wt{m}_\e,m_\e\rightharpoonup m_0   \quad\text{weakly in }H^1(\domain; \hspace{0.03cm}\Rt).
 \end{align*}
 Since $m_\e$ converges uniformly to $m_0$ and $\nabla\pi_\M$ is uniformly continuous as a continuous function on the compact set $\ov{U}_\delta$, we also infer
 \begin{align}
 \notag
     \nabla\pi_\M[\wt{m}_\e]\rightarrow\nabla\pi_\M[m_0] 
     \quad\text{strongly in }L^\infty(\domain; \hspace{0.03cm}\Rtt).
 \end{align}
 On the one hand, we deduce that, for all $j,k,\ell=1,\ldots,3$, $i=1,\ldots,n$,
 \begin{align*}
     \e \big(b_{i,\ell}^\intercal(T_{x/\e}\wt{\omega})\partial_{x_j}\vphi_{i,k\ell}(x)\big)_{jk}
     \rightarrow
     0
     \quad\text{strongly in }L^\infty(\domain).
 \end{align*}
 On the other hand, since $\tomega$ is a typical trajectory, we find 
 \begin{align*}
     \nablaomega b_i(T_{x/\e}\wt{\omega})\vphi_i^\intercal(x)
     \sts
     \nablaomega b_i(\omega)\vphi_i^\intercal(x)
     \quad\text{strongly in }L^2(\domain\times\Omega; \hspace{0.03cm}\Rtt).
 \end{align*}
  where we have applied the mean value property (see Proposition \ref{prop:mean-value property}) to 
 $\nablaomega b_i(T_{x/\e}\wt{\omega})\vphi_i^\intercal(x)$ for the weak and to 
 $|\nablaomega b_i(T_{x/\e}\wt{\omega})|^2|\vphi_i(x)|^2$ for the strong stochastic 2-scale convergence. 
 Hence, since  $\nabla m_0\nabla\pi_\M[m_0] = \nabla m_0$ since $\pi_\M[m_0]=m_0$, we conclude that 
 \begin{align*}
     \nabla m_\e(x) 
     \sts 
     \nabla m_0(x)+\sum_{i=1}^n \nablaomega b_i(\omega)\vphi_i^\intercal(x)\nabla\pi_\M[m_0(x)]
     \quad\text{strongly in }L^2(\domain\times\Omega; \hspace{0.03cm}\Rtt).
 \end{align*}
Applying Proposition \ref{prop:weak-strong conv} and Remark \ref{rem:typicaltrajectories} yields
 \begin{align}
     &\lim_{\e\rightarrow 0}\G_\e(m_\e)
     =
     \lim_{\e\rightarrow 0}\bigg(\frac{1}{2} \int_\domain a(T_{x/\e}\tomega)|\nabla m_\e(x)|^2\dx
     -\int_\domain\kappa(T_{x/\e}\tomega)\bchi(m_\e(x))\colon\nabla m_\e(x)\dx\bigg)\notag\\
     &\quad=
     \frac{1}{2} \int_{\domain\times\Omega} a(\omega)\bigg|\nabla m_0(x)+\sum_{i=1}^n \nablaomega b_i(\omega)\vphi_i^\intercal(x)\nabla\pi_\M[m_0(x)]\bigg|^2\dPo \dx\notag\\
     &\qquad-\int_{\domain\times\Omega} \kappa(\omega)\bchi(m_0(x))
        \colon\bigg(\nabla m_0(x)+\sum_{i=1}^n \nablaomega b_i(\omega)\vphi_i^\intercal(x)\nabla\pi_\M[m_0(x)]\bigg)\dPo \dx.\label{eq:limsup:limGeme}
 \end{align}

 \underline{Step 2: Recovery sequence for perturbations of $m_0$ with the correctors.} 
 Let $m_0\in H^1(\domain;\hspace{0.03cm}\M)$. 
 Based on Proposition \ref{prop:ThomGhom} and \eqref{eq:ThomMinimzer}, we want the perturbations to result in the gradient correction 
 \begin{equation}\label{eq:limsup:defXi}
     \Xi(x,\omega) \coloneqq 
     \Theta_a(\omega)\nabla m_0(x) - \Theta_\kappa(\omega)\bchi(m_0(x))\pi_{\T[s]\M}
 \end{equation}
 with $\Xi(x,\cdot)\in\Ltpot(\Omega;\hspace{0.03cm}\T[m_0(x)]\M)$ for almost every $x\in\domain$.
 Let $\{\vphi_{a,\delta}\}_{\delta>0},\{\vphi_{\kappa,\delta}\}_{\delta>0}\subset C^\infty_c(\domain;\Rtt)$ be such that
 \begin{align*}
     \vphi_{a,\delta}^\intercal         &\xrightarrow{\delta\rightarrow 0}      \nabla m_0 
     \quad\text{strongly in }L^2(\domain;\hspace{0.03cm}\Rtt),\\
     \vphi_{\kappa,\delta}^\intercal    &\xrightarrow{\delta\rightarrow 0}      \bchi(m_0(x))\pi_{\T[s]\M}
     \quad\text{strongly in }L^2(\domain;\hspace{0.03cm}\Rtt).\\
 \intertext{By the definition of $\Ltpot(\Omega;\Rt)$ there exist sequences $\{b_{a,\delta}\}_{\delta>0},\{b_{\kappa,\delta}\}_{\delta>0}\subset C^1(\Omega;\Rt)$ with }
     \nablaomega b_{a,\delta}           &\xrightarrow{\delta\rightarrow 0}      \Theta_a 
     \quad\text{strongly in }\Ltpot(\Omega;\hspace{0.03cm}\Rt),\\
     \nablaomega b_{\kappa,\delta}      &\xrightarrow{\delta\rightarrow 0}      \Theta_\kappa 
     \quad\text{strongly in }\Ltpot(\Omega;\hspace{0.03cm}\Rt).
 \end{align*}
 To summarize,
 \begin{align}
   \notag     \Psi_\delta\coloneqq 
     \nablaomega b_{a,\delta}\vphi_{a,\delta}^\intercal
     +\nablaomega b_{\kappa,\delta}\vphi_{\kappa,\delta}^\intercal
     \xrightarrow{\delta\rightarrow 0} \Xi(x,\omega) 
     \quad\text{strongly in }L^2(\domain;\hspace{0.03cm}\Ltpot(\Omega;\hspace{0.03cm}\Rt)).
     %\label{limsupin1}
 \end{align}
 In the spirit of Step 1, for fixed $\delta>0$, we define the sequence $\{m_\e^\delta\}_{\e>0}$ by
 \begin{align*}
     m_\e^\delta \coloneqq \pi_\M\big[m_0+\e(\vphi_{a,\delta}b_{a,\delta}+\vphi_{\kappa,\delta}b_{\kappa,\delta})\big].
 \end{align*}
 According to \eqref{eq:limsup:limGeme}, we have that 
 \begin{align*}
     \lim_{\e\rightarrow 0}\G_\e(m_\e^\delta)
     = F(m_0,\Psi_\delta),
 \end{align*}
 where $F(m_0,\Psi)$ for $\Psi\in L^2(D; L^2_\pot(\Omega;\hspace{0.03cm} \Rt))$ is given by
 \begin{multline*}
     F(m_0,\Psi)\coloneqq 
     \frac{1}{2} \int_{\domain\times\Omega} a(\omega)\bigg|\nabla m_0(x)+\Psi(x,\omega)\nabla\pi_\M[m_0(x)]\bigg|^2\dPo \dx\\
     -\int_{\domain\times\Omega} \kappa(\omega)\bchi(m_0(x))
        \colon\bigg(\nabla m_0(x)+\Psi(x,\omega)\nabla\pi_\M[m_0(x)]\bigg)\dPo \dx.
 \end{multline*}
 Additionally, from Proposition \ref{prop:ThomGhom} and by the definition of $\Xi$ in \eqref{eq:limsup:defXi} we have 
 \begin{align}
 \notag
     F(m_0,\Xi)=\Ghom(m_0).
 \end{align}
 This equality is a consequence of the fact that $\Xi(x,\omega)\nabla\pi_\M[m_0(x)]=\Xi(x,\omega)$, which we prove in Lemma \ref{lem:limsup:TM piM interaction} below. 
 Now, due to the strong $L^2$-convergence of $(\Psi_\delta)_{\delta>0}$, there holds
 \begin{align*}
     \lim_{\delta\rightarrow 0} F(m_0,\Psi_\delta)
     =F(m_0,\Xi)=\Ghom(m_0).
 \end{align*}
The thesis follows then by Step 1.
 \end{proof}

\begin{lemma}\label{lem:limsup:TM piM interaction}
Let $\M$ be a $C^2$ orientable hypersurface of $\Rt$ with a tubular neighborhood $U_\delta$ of size $\delta$.
Let $\pi_\M\in C^1(U_\delta;\M)$ be the pointwise projection onto the manifold.
Then for every $s\in\M$ and every $\tau\in\T_s\M$,
\begin{align*}
    \tau^\intercal\nabla\pi_\M[s]=\tau^\intercal.
\end{align*}
\end{lemma}

\begin{proof}
Let $s\in\M$ and $\tau\in\T_s\M$.
By the definition of the tangent space and in view of the regularity of $\M$, we have that 
\begin{align*}
    \pi_\M(s+h\tau) = s+h\tau + \mathcal{O}(h^2)
    \quad\text{as }h\rightarrow 0.
\end{align*}
Now, the claim follows from the following equality
\begin{align*}
    \tau^\intercal\nabla\pi_\M[s]
    =\frac{\mathrm{d}}{\mathrm{d}h}\bigg(\pi_\M(s+h\tau)\bigg|_{h=0}\bigg)^\intercal
    = \tau^\intercal,
\end{align*}
which concludes the proof.
\end{proof}

\section{Proof of the main result}\label{section:ProofMainResult}
This section is devoted to completing the proof of Theorem \ref{mainthm}. 
With Theorem \ref{thm:GeGammaConv} proven in Section \ref{section:GammaresultExDMI}, in this section it remains to prove 
    Proposition \ref{prop:intro:Whom}, i.e. the characterization of the limiting behaviour of the magnetostatic self-energy functionals $\W_\e$ (Subsection \ref{section:demagnefield}),
    Proposition \ref{prop:intro:AhomZehom}, i.e. the study of the asymptotic behavior of the anisotropic and Zeeman energy (Subsection \ref{section:Aniso+Zeeman}), 
    and the equi-coerciveness of the overall family $(\F_\e)_{\e>0}$. 
%First, note that Theorem \ref{thm:Gammaconv ex+DMI} holds in the micromagnetic setting, i.e., $\M=\St$. Hence, in the following subsections, we will focus on the characterization of the limiting behaviour of the magnetostatic self-energy functionals $\W_\e$, the anisotropy energy functionals $\A_\e$, as well as, the interaction energies $\Ze_\e$. 

\subsection{Homogenization of the demagnetizing field}
\label{section:demagnefield}
In this section, we show that the family of magnetostatic self-energy functionals $\W_\e$ continuously converges to the energy $\W_\hom$ defined in \eqref{eq:intro:def:Whom}.  
To this end,  we recall the definition of Beppo-Levi spaces. 
Set $\theta(x) := (1+|x|^2)^{-1/2}$ and let $L^2_\theta (\R^3)$ be the weighted Lebesgue space
    $L^2_\theta(\R^3) := \{u\in {\mathcal{D}}'(\R^3)\hspace{0.03cm}:\hspace{0.03cm} u\theta\in L^2(\R^3) \}$. 
The Beppo-Levi space $\BL(\R^3)$ is defined as
    \begin{equation}
      \notag
      \BL(\R^3):= \left\{ u\in {\mathcal{D}}'(\R^3) \hspace{0.03cm} : \hspace{0.03cm} u\in L^2_\theta (\R^3) \quad\mbox{and}\quad \nabla u \in L^2(\R^3; \hspace{0.03cm} \R^3) \right\}.
    \end{equation}
The space $\BL(\Rt)$ endowed with $(u, v)_{\BL(\Rt)}:= (\nabla u, \nabla v)_{L^2(\Rt)}$ is a Hilbert space due to the Hardy inequality. \\ 
Recall that a given magnetization $m\in L^2(\Rt; \hspace{0.03cm} \Rt)$ generates the stray field $h_d[m] = \nabla u_m$, where the potential $u_m$ solves 
       \begin{equation}
           \label{prob um}
           \Delta u_m = -\div (m) \quad \mbox{in } {\mathcal{D}}'(\Rt).
       \end{equation} 
By Lax-Milgram's theorem, there exists a unique solution to the variational formulation associated to \eqref{prob um}: namely, a potential $u_m\in \BL(\Rt)$ such that, for all $\varphi\in\BL(\Rt)$,
    \begin{equation}
    \label{varformum}
        (u_m, \varphi)_{\BL(\Rt)} := \int_{\Rt} \nabla u_m\cdot\nabla\varphi \dx = -\int_{\Rt} m\cdot \nabla \varphi \dx.
    \end{equation}
Additionally, we have
    \begin{equation}
    \label{stabilityesti}
        \|u_m\|_{BL(\Rt)}\equiv \|\nabla u_m\|^2_{L^2(\Rt,\Rt}) \leq \sup_{\substack{ \varphi\in\BL(\Rt),\\ \|\varphi\|_{L^2(\Rt)}=1}} \left|\int_{\Rt} m\cdot \nabla \varphi\, dx \right|\leq \|m\|_{L^2(\Rt)}.
    \end{equation}
The family of magnetostatic self-energy functionals $\W_\e: L^2(D; \hspace{0.03cm} \M)\to \R$ is given by 
   \begin{equation}
      \notag %\label{def:magnetoselfene}
       \W_\e (m):= -\frac{\mu_0}{2}\left(h_d\left[\Ms\left(T_{\cdot/\e}\tomega\right) m\chi_D\right], \Ms\left(T_{\cdot/\e}\tomega\right) m\right)_{L^2(D)}
   \end{equation}
for $\tomega\in\wt{\Omega}$ as in Remark \ref{rem:typicaltrajectories}, where $m\chi_D$ denotes the extension of $m$ to $\Rt$ vanishing outside $D$.
We will prove the $\Gamma$-continuous convergence of $\W_\e$ via an application of Proposition \ref{prop:weak-strong conv}, the weak-strong convergence principle for stochastic two-scale convergence.
However, we will first need to transfer some basic stochastic two-scale convergence results to the weighted spaces $L^2_\theta$ and $\BL$. 
In particular, we provide compactness results for bounded sequences in these spaces, which are an adaption of \cite[Propositions 4.1 and 4.2]{ADF15} to the stochastic setting.

\begin{prop}[Weighted stochastic 2-scale compactness in $L^2$]
\label{prop:stoc2scL2theta}
Let $\{u_\e\}_{\e>0}$ be a bounded sequence in $L^2_\theta(\Rt)$. Let $\wt{\omega}\in\wt{\Omega}$. 
Then, there exists a function $u\in L^2_\loc(\Rt\times \Omega)$ such that $\Ev[u]\in L^2_\theta(\Rt)$ and, up to subsequences, 
\begin{equation*}
    \lim_{\e\to 0} \int_{\Rt} u_\e(x)\vphi(x) b(T_{x/\e}\wt{\omega})\dx
    = \int_{\Rt\times\Omega} u(x,\omega)\vphi(x)b(\omega)\dPo \dx,
\end{equation*}
for all $\vphi\in \Cic(\Rt)$ and $b\in C(\Omega)$. 
\end{prop}
We say that a sequence $\{u_\e\}_{\e>0}$ as in Proposition \ref{prop:stoc2scL2theta} stochastically $L^2_\theta$-two-scale converges to $u$.

\begin{proof}%[Proof of Proposition \ref{prop:stoc2scL2theta}]
The boundedness of $\{u_\e\}_\e$ in $L^2_\theta(\Rt)$ implies the existence of $u_\infty\in L^2_\theta(\Rt)$ and of a subsequence $\{u_{\e_n}\}_{n\in\N}$ such that 
     \begin{equation}
         \notag
         u_{\e_n}\rightharpoonup u_\infty \quad \mbox{weakly in } L^2_\theta(\Rt).
     \end{equation}
In particular, for every bounded domain $D\in\Rt$,  we have that $u_{\e_n}\rightharpoonup u_\infty$ weakly in $L^2(D)$.\\ 
Let $(D_i)_{i\in\mathbb{N}}$ be a sequence of bounded domains covering $\Rt$. 
An application of Theorem \ref{thm:L2cpt} to domain $D_1$  yields to the existence of a subsequence $\{u_{\e_{n(k_1)}}\}$ of $\{u_{\e_n}\}$  and $u_1\in L^2(D_1\times \Omega)$ such that 
   \begin{equation}
       \notag
       u_{\e_{n(k_1)}}\wsts\,  u_1 \quad \mbox{in } L^2(D_1\times \Omega).
   \end{equation}
This subsequence $\{u_{\e_{n(k_1)}}\}$ is again bounded in $L^2(D_2)$, hence there exists a subsequence $\{u_{\e_{n(k_2)}}\}$ such that $u_{\e_{n(k_2)}}$ stochastic two-scale converges to some $u_2\in L^2(D_2\times \Omega)$. 
In addition, the uniqueness of the two-scale limits implies that $u_1$ and $u_2$ coincide on $(D_1\cap D_2)\times \Omega$. 
Repeating the same arguments for any $D_i$, with $i\in\N$, we conclude the existence of a subsequence $\{u_{\e_{n(k_i)}}\}\subset\{u_{\e_{n(k_{i-1})}}\}$ such that $u_{\e_{n(k_i)}}$ stochastic two-scale converges to some $u_i\in L^2(D_i\times\Omega)$. 
This allows us to define a diagonal sequence of indices 
    \begin{equation}
        \label{def:seqindices}
        n(k_\infty(1))=n(k_1(1)),\hspace{0.1cm} n(k_\infty(2))=n(k_2(2)), \dots, n(k_\infty(i))=n(k_i(i)), \dots.
    \end{equation}
 Since for any $i\in\N$, up to the first $(i-1)$ terms, the sequence of indices $n(k_\infty)$ is included in $n(k_i)$, we deduce that $u_{\e_{n(k_\infty)}}$ stochastic two-scale converges to $u_i$ in $L^2(D_i\times\Omega)$ for all $i\in\N$. 
 Once again, the uniqueness of the two-scale limits implies that $u_i \equiv u_j$ on $(D_i\cap D_j)\times\Omega$ whenever $D_i\cap D_j\neq\emptyset$. 
 Therefore, thanks to the  {\it principle du recollement des morceaus} (see, e.g., \cite{S66}), we deduce that there exists a unique distribution $u\in L^2_\loc(\Rt\times \Omega)$ such that $u$ coincides with $u_i$ on $D_i\times\Omega$ and, for $\varphi\in C^\infty_c(D)$ and $b\in C(\Omega)$, 
     \begin{equation}
         \notag        
         \lim_{k_\infty\to\infty}\int_{\Rt}u_{\e_{n(k_\infty)}}(x)\varphi(x)b(T_{x/ \e}\omega)\dx =\int_{\Rt\times\Omega} u(x, \omega) \varphi(x)b(\omega) b\dx \dPo .
     \end{equation}
Additionally, $\{u_{\e_{n(k_\infty)}}\}$ weakly converges to $\Ev(u)$ in $L^2(D_i)$ for every $i\in\N$, see Remark \ref{rem:2sL2vsL2}.
Hence, since $\{u_{\e_n}\}$ and thus $\{u_{\e_{n(k_\infty)}}\}$ weakly converges in $L^2(D)$ to $u_\infty$ for every bounded domain $D\in\Rt$,
we deduce that $\Ev(u)\equiv u_\infty\in L^2_\theta(\Rt)$.
This concludes the proof.
\end{proof}

\begin{prop}[Stochastic 2-scale compactness in $\BL$]
\label{prop:cptBL}
Let $\{u_\e\}_{\e>0}$ be a bounded sequence in $\BL(\Rt)$ which weakly converges to $u_\infty$.
Then, $u_\e$ stochastically $L^2_\theta$-two-scale converges to some $u_\infty\in\BL(\Rt)$ and there exists $\xi\in L^2(\Rt;\hspace{0.03cm}\Ltpot(\Omega))$ such that, up to a subsequence, 
\begin{align}
\label{convergenceBL}
    \nabla u_\e \wsts \nabla u_0(x) + \xi(x,\omega) 
    \quad\text{weakly in }L^2(\Rt\times\Omega;\Rt).
\end{align}
\end{prop}
    \begin{proof}
        First, note that the weak convergence of $u_\e$  in $BL^1(\Rt)$ implies the weak convergence of $u_{\e}$ in $L^2_\theta(\Rt)$ to $u_\infty$. Hence, an application of Proposition \ref{prop:stoc2scL2theta} yields the existence of a function $u\in L^2_{\loc}(\Rt\times Q)$ such that, up to a subsequence, $u_\e$ stochastically $L^2_\theta$-two scale converges to $u  \in L^2_{\loc}(\Rt\times\Omega)$. Repeating similar arguments as those in Proposition \ref{prop:stoc2scL2theta}, we find that, for every $i\in\mathbb{N}$, there exists a subsequence $\{u_{\e_{n(k_i)}}\}$ such that $n(k_i)\subset n(k_{i-1})$ and  
            \begin{equation}
                %\label{eq1}
                \notag
                u_{\e_{n(k_i)}} \wsts u_\infty(x) \equiv \Ev[u] \equiv u(x, \omega) \quad\mbox{in } L^2(D_i\times \Omega).
            \end{equation}
        Here, $\{D_i\}_{i\in\mathbb{N}}$ is a covering of $\Rt$ consisting of bounded domains. Defining the diagonal sequence of indices as in \eqref{def:seqindices}, for every $i\in\mathbb{N}$, up to the first $(i-1)$ terms, the sequence of indices $n(k_\infty)$ is included in $n(k_i)$ and 
            \begin{equation}
                \notag
                u_{\e_{n(k_\infty)}} \wsts u_\infty(x) \equiv \Ev[u] \equiv u(x, \omega) \quad\mbox{in } L^2(D_i\times \Omega).
            \end{equation}
       In particular, this implies that $u\equiv u_\infty\in L^2_\theta (\Rt)$ in $\Rt$.
       \par To conclude the proof, it remains to show \eqref{convergenceBL}. For simplicity, with a slight abuse of notation, from now on we omit the index $n(k_\infty)$ and simply denote the above exctracted subsequence by $\ep$.
       Once again, the weak convergence of  $u_\e$ in $BL^1(\Rt)$ to $u_\infty$ implies that $\nabla u_\e$ weakly converges to $\nabla u_\infty$ in $L^2_\theta(\Rt)$, and that $\nabla u_\e$ is bounded in $L^2(\Rt)$. 
       Due to Theorem \ref{thm:L2cpt}, there exists $\lambda_\infty\in L^2(\Rt\times\Omega;\hspace{0.03cm} \Rt)$ such that, up to subsequences, $\nabla u_{\e}$ stochastically two-scale converges to $\lambda_\infty$. Let $\varphi\in C^\infty_c(\Rt)$. Integrating by parts, we deduce that 
            \begin{align}
                \notag
                \int_{\Rt} \nabla u_{\e}(x)\varphi(x)\dx = - \int_{\Rt} u_{\e}(x)\nabla\varphi(x)\dx.
            \end{align}
        Passing to the two-scale limit, we get 
             \begin{align}
                \notag
                \int_{\Rt} \varphi(x) \left[\int_\Omega \lambda_\infty (x, \omega)\dPo\right]\dx = - \int_{\Rt} u_\infty(x)\nabla\varphi(x)\dx.
             \end{align}
      Therefore, we define the distribution $\nabla u_\infty$  by 
          \begin{equation}
              \notag
              \nabla u_\infty  := \int_\Omega  \lambda_\infty (\cdot, \omega)\dPo \in L^2(\Rt).
          \end{equation}
         Now, let $\sigma \in C(\Omega;\hspace{0.03cm} \Rt)\cap L^2_\sol (\Omega)$ and $\varphi\in C^\infty_c(\Rt)$. 
         Hence, 
             \begin{align}
                 \int_{\Rt} \nabla u_{\e}\cdot \sigma (T_{x/\e}\omega)\varphi(x)\dx & =\int_{\Rt}\nabla [u_{\e}\varphi(x)] \cdot \sigma (T_{x/\e}\omega)\dx\notag\\
                 & \quad -\int_{\Rt} u_{\e}(x) \nabla \varphi(x) \cdot  \sigma (T_{x/\e}\omega)\dx\notag\\
                 & = - \int_{\Rt} u_{\e}(x) \nabla \varphi(x) \cdot  \sigma (T_{x/\e}\omega)\dx\notag.
             \end{align}
          Taking the limit as $\e\to 0$ and integrating by parts, we find
              \begin{align}
                  \int_{\Rt\times \Omega} \lambda_\infty (x, \omega)\cdot \sigma(\omega)\varphi(x)\dPo\dx & = - \int_{\Rt\times \Omega} u_\infty(x) \nabla \varphi(x) \cdot  \sigma (\omega)\dPo\dx \notag\\
                  & = \int_{\Rt\times \Omega} \nabla u_\infty(x) \cdot  \sigma (\omega) \varphi(x) \dPo\dx. \notag
              \end{align}
        The last equality shows that $\xi(x, \omega):=\lambda_\infty(x, \omega)-\nabla u_\infty(x)\in (C(\Omega;\hspace{0.03cm} \Rt)\cap L^2_\sol(\Omega))^\perp$ for almost every $x\in \Rt$. 
        By the density of $C(\Omega)$ in $L^2(\Omega)$, we conclude that  $\xi\in L^2(\Rt;\hspace{0.03cm} (L^2_\sol(\Omega))^\perp) = L^2(\Rt;\hspace{0.03cm} L^2_\pot(\Omega))$.
    \end{proof}

With this result at hands, we are now in a position to characterize the behaviour of the stray fields $\{\hd[m_\ep]\}_{\e>0}$ defined in \eqref{prob um}.
Applying the characterization to the sequence $\Ms(T_{\cdot/\e} \tomega)m\chi_\domain$ for $m\in L^2(\domain;\St)$ as in the definition of $\W_\e$, will then allow us to prove Proposition \ref{prop:intro:Whom}.

\begin{prop}[Stochastic 2-scale convergence and the demagnetizing field]
\label{prop:2scaleconvdemagne}
Let $\{m_\e\}_{\e>0}$ be a bounded family in $L^2(\Rt;\hspace{0.03cm}\Rt)$ which stochastically 2-scale converges to $m=m(x,\omega)\in L^2(\Rt\times\Omega;\hspace{0.03cm}\Rt)$.
Then, the stochastic 2-scale limit of $\{\hd[m_\e]\}_{\e>0}\subset L^2(\Rt;\hspace{0.03cm}\Rt)$ exists and is given by
\begin{align*}
    \hd(x,\omega)= \hd\big[\Ev[m]\big](x) + \xi_m(x, \omega),
\end{align*}
where, for a.e. $x\in\Rt$, the map $\xi_m(x,\cdot)$ is the unique solution in $\Ltpot(\Omega)$ to
\begin{align}
\label{auxprobl}
    m(x,\cdot)+\xi_m(x,\cdot)\in\Ltsol(\Omega).
\end{align}  
\begin{remark}
    The variational formulation of \eqref{auxprobl} reads
        \begin{equation}
            \notag%\label{weakform3}
            \int_{\Omega} \xi_m(x, \omega)\cdot \Psi(\omega)\dPo= - \int_{\Omega} m(x, \omega)\cdot \Psi(\omega)\dPo
        \end{equation}
for any $\Psi\in L^2_\pot(\Omega)$.
\end{remark}
\end{prop}
\begin{proof}
     Combining the stability estimate \eqref{stabilityesti} with the boundedness in $L^2(\Rt)$ of $m_\e$, the sequence of magnetostatic potentials $\{u_{m_\e}\}_{\e>0}$  is uniformly bounded in $BL^1(\Rt)$.
    Thus, there exists $u_m\in BL^1(\Rt)$ such that, up to the extraction of a non-relabelled subsequence, $u_{m_\e}$ weakly converges to $u_m$ in $BL^1(\Rt)$. 
    In view of Proposition \ref{prop:cptBL}, there exist functions $u_m\in BL^1(\Rt)$ and $\xi_m\in L^2(\Rt; \hspace{0.03cm} L^2_\pot(\Omega))$ such that 
        \begin{equation}
            \notag
            u_{m_\e}\wsts u_m \quad\mbox{in } L^2_\theta \quad \mbox{and}\quad \nabla u_{m_\e}\wsts \nabla u_m + \xi_m\quad \mbox{in } L^2(\Rt\times \Omega;\Rt).
        \end{equation}
   Testing  \eqref{varformum} with $\varphi(x) + \e \psi(x)b(T_{x/\e}\omega)$, for $\varphi, \psi \in C^\infty_c(\Rt)$ and $b\in C^1(\Omega)$, we infer 
        \begin{multline*}
            \int_{\Rt} \nabla u_{m_\e} \cdot (\nabla \varphi (x) + \psi(x)\nablaomega b(T_{x/\e}\omega) + \e b(T_{x/\e}\omega)\nabla \psi(x) ) \dx \\
            = -\int_{\Rt} m_\e \cdot (\nabla\varphi(x)  +\psi(x) \nablaomega  b(T_{x/\e}\omega)+ \e b(T_{x/\e}\omega)\nabla \psi(x)) \dx.
        \end{multline*}
    Passing to the stochastic two-scale limit yields 
         \begin{multline}\label{eq:hd:testedLimit}
             \int_{\Rt\times\Omega}  (\nabla u_m(x)+\xi_m(x, \omega) ) \cdot  (\nabla \varphi (x)  + \psi(x)\nablaomega b(
             \omega))\dPo \dx \\
             =  - \int_{\Rt\times \Omega} m(x, \omega)\cdot (\nabla\varphi(x)  +\psi(x) \nablaomega b(\omega) ) \dPo \dx.
         \end{multline}
         In particular, choosing $\psi\equiv 0$ we have  
              \begin{align}
              \notag
                   -\int_{\Rt}\Ev[m(x,\cdot)]\cdot  \nabla \varphi (x)  \dx 
                   &=\int_{\Rt\times \Omega} (\nabla u_m(x)+\xi_m(x, \omega) )\cdot \nabla \varphi (x)\dx\\
                   &= \int_{\Rt} \nabla u_m(x)\cdot \nabla \varphi (x)\dx,\notag
              \end{align}
        where the last equality is a consequence of Lemma \ref{lemma:stochintegrationparts}.
        %Therefore, the weak limit $u_m$  satisfied the variational formulation \eqref{weakfor1}. J:not sure if this is the correct reference.
        In other words, $\nabla u_m = \hd\left[\Ev[m]\right]$, since $u_m\in \BL(\Rt)$ is a solution to 
            \begin{equation*}
                - \div \left(\nabla u_m + \Ev[m]\right)  = 0.
            \end{equation*}
      Now, choosing $\varphi\equiv 0$, \eqref{eq:hd:testedLimit} yields
            \begin{align*}
               \int_{\Rt\times\Omega} (\nabla u_m(x) + \xi_m(x,\omega))\cdot\nabla_\omega b(\omega) \psi(x)\dPo\dx
               = -\int_{\Rt\times \Omega} m(x, \omega) \cdot \nabla_\omega b(\omega) \psi(x)\dPo\dx. 
            \end{align*}
    This implies that for almost every $x\in\Rt$
        \begin{align}
            -\int_{\Omega} m(x, \omega) \cdot \nabla_\omega b(\omega)\dPo &=\int_{\Omega} (\nabla u_m(x) + \xi_m(x,\omega))\cdot\nabla_\omega b(\omega) \dPo\notag\\
            &= \int_{\Omega} \xi_m(x,\omega)\cdot\nabla b(\omega) \dPo, \label{weakfor2}
        \end{align}
    where in the last equality we  used again Lemma  \ref{lemma:stochintegrationparts}.  
    The equality \eqref{weakfor2} is the variational formulation of 
        \begin{equation}
            \notag
            \div_{\omega} \xi_m (x, \omega) = -\div_\omega m(x, \omega) \quad \mbox{in } L^2_\pot(\Omega).
        \end{equation}
        The well-posedness of this problem is a consequence of Lax-Milgram's theorem in $L^2_\pot(\Omega)$. 
        This concludes the proof.
\end{proof}

\begin{comment}%Replaced by Proposition \ref{prop:intro:Whom}
    \begin{prop}\label{prop:WhomConvergence}
    The family of magnetostatic self-energies $\W_\e$ defined by \eqref{def:magnetoselfene}
    $\Gamma$-continuously converges to the functional $\W_\hom:  L^2(\domain;\hspace{0.03cm}\St) \to \R$ given by 
    \begin{align}
        \W_\hom(m):= -\Ev[\Ms]^2\big(\hd[m],m\big)_{L^2(\domain)}
        +\|\xi_m\|^2_{L^2(\domain\times\Omega)} \label{def:Whom}
    \end{align}
    where, for a. e. $x\in\domain$, the function is the unique solution of 
    \begin{align*}
        \xi_m(x,\cdot)\in\Ltpot(\Omega),
        \quad m(x)\Ms(\cdot)-\xi_m(x,\cdot)\in \Ltsol(\Omega).
    \end{align*}
    \end{prop}
\end{comment}

We conclude this subsection with the proof of Proposition \ref{prop:intro:Whom}.
\begin{proof}[Proof of Proposition \ref{prop:intro:Whom}]
Let $m\in L^2(D;\St)\subset L^2(\Rt;\Rt)$ and fix $\tomega\in\wt{\Omega}$, see Remark \ref{rem:mean-value property}.
Define $M_\e\in L^\infty(\Rt)$ by $M_\e(x)\coloneqq \Ms(T_{x/\e}\tomega)$.
Then, the sequence $\{M_\e m\}$ weakly stochastically two-scale converges to $\Ms(\omega) m(x)$ componentwise. 
Since $\tomega$ is a typical trajectory with respect to  $\Ms^2\dP$, and $m(x)\in\St$ for every $x\in\domain$,
    \begin{equation*}
        \|M_\e m\|_{L^2(\domain)}^2
        =\int_\domain \Ms(T_{x/\e}\tomega)^2\dx 
        \overset{\e\rightarrow 0}{\longrightarrow}
        \|\Ms m\|_{L^2(\domain\times\Omega)}^2.
    \end{equation*}
    Thus, $\{M_\e m\}$ strongly stochastically two-scale converges to $\Ms m$.
    Hence, we can apply Proposition \ref{prop:2scaleconvdemagne} for the corresponding stray fields.
    Note that by the definition of $\Theta_M\in\Ltpot(\Omega;\Rt)$, cf. \eqref{eq:correctors:ThetaM},  for almost every $x\in D$ we have
    \begin{equation}\label{eq:Whom:corrector}
        \Ms m(x) + \Theta_M m(x) \in\Ltsol(\Omega)
    \end{equation}
    and therefore
    \begin{equation*}
        \hd\left[M_\e m \chi_\domain\right] \wsts \hd\left[\Ev[\Ms]m\right] + \Theta_M m
        \quad\text{in }L^2(\Rt\times\Omega;\Rt).
    \end{equation*}
    Thus, the weak-strong convergence principle in Proposition \ref{prop:weak-strong conv} yields
    \begin{align*}
        \lim_{\e\to 0}\W_\e(m) 
        &=  -\lim_{\e\to 0}\frac{\mu_0}{2}(h_d[M_\e m], M_\e m)_{L^2(D)}\\
        &=  -\frac{\mu_0}{2}\Ev\left[\left(\hd\left[\Ev[\Ms]m\right],\Ms m\right)_{L^2(D)} + \left(\Theta_M m,\Ms m\right)_{L^2(\domain)} \right]\\
        &=  -\frac{\mu_0}{2}\left(\hd\left[\Ev[\Ms]m\right],\Ev[\Ms] m\right)_{L^2(\domain)} 
            -\frac{\mu_0}{2}\Ev\left[\left(\Theta_M m,\Theta_M m\right)_{L^2(\domain)} \right]\\
        &=  \Whom(m)
    \end{align*}
    where we used \eqref{eq:Whom:corrector} with $\Theta_M m(x)\in\Ltpot(\Omega)$ for almost every $x\in\domain$.
    To conclude the proof, it remains to show that the family $\W_\e$ continuously converges to $\Whom$ in $L^2$, i.e. that for $m_\e\to m_0$ strongly in $L^2(\domain;\St)$ 
    there holds
       \begin{align}
       \label{contlimdegmselfener}
            \W_\e (m_\e) \rightarrow \Whom (m_0) \quad\text{as }\e\rightarrow 0.
           %\lim_{(m, \e) \to (m_0, 0)} \W_\e (m) = \Whom (m_0). 
       \end{align}
        To this end, we write 
           \begin{align}
           \notag
               |\W_\e (m_\e) - \Whom (m_0)| \leq |\W_\e (m_\e) - \W_\e (m_0)| +|\W_\e (m_0) - \Whom (m_0)|.
           \end{align}
    Due to Proposition \ref{prop:2scaleconvdemagne}, it follows that $|\W_\e (m_0) - \Whom (m_0)| \to 0$ as $\e\to 0$.  
    Since $\{m_\e\}_{\e>0}\subset L^2(\domain;\St)$, by  \eqref{stabilityesti} we infer
         \begin{align}
             |\W_\e (m_\e) - \W_\e (m_0)| & \leq |(h_d[M_\e m_\e], M_\e (m_\e-m_0))_{L^2(D)} + (h_d[M_\e (m_\e-m_0)], M_\e m_0)_{L^2(D)}|\notag\\
             & \leq 2 \|\Ms\|^2_{L^\infty (D)} \|m_\e-m_0\|_{L^2(D; \R^3)},\notag
         \end{align}
  which in turn yields \eqref{contlimdegmselfener}.
\end{proof}
\subsection{The homogenized anisotropy and the interaction energies}
\label{section:Aniso+Zeeman}
In this subsection we prove Proposition \ref{prop:intro:AhomZehom}, i.e. we focus on the continuous convergence of the family of anisotropy energy functionals $\{\A_\e\}$ and of the family of Zeeman energy functionals $\{\Ze_\e\}$.

\begin{proof}[Proof of Proposition \ref{prop:intro:AhomZehom}]
\underline{Step 1: The convergence of $\{\A_\e\}$.}
    We prove that for every $m_0\in L^2(D; \hspace{0.03cm}\mathbb{S}^2)$ and every sequence $\{m_\e\}\subset L^2(D; \hspace{0.03cm}\mathbb{S}^2)$ with $m_\e\to m_0$ strongly in $L^2(D; \hspace{0.03cm}\mathbb{S}^2)$ there holds
    \begin{equation}
       \label{eq8a}
       \A_\e (m_\e) \rightarrow \Ahom (m_0) \quad\text{as }\e\rightarrow 0.
    \end{equation}
    In fact, we have
     \begin{equation}
         \label{eq7a}
         |\A_\e(m_\e) -\Ahom(m_0)|\leq |\A_\e(m_\e) -\A_\e(m_0)| + |\A_\e(m_0) -\Ahom(m_0)|.
     \end{equation}
    Now, we estimate the two terms on the right-hand side of \eqref{eq7a}. 
    An application of Birkhoff's ergodic theorem combined with an approximation argument yields $|\A_\e(m_0) -\Ahom(m_0)|\to 0$, as $\e\to 0$ (see also Remark \ref{rk:Birkhoff}).
    On the other hand, using the global Lipschitz continuity of $\varphi$ combined with H\"{o}lder inequality, we deduce
     \begin{align}
         |\A_\e(m_\e) -\A_\e(m_0)| &\leq \int_D\left| \varphi\left(T_{x/\e}\omega, m_\e\right) - \varphi\left(T_{x/\e}\omega, m_0\right)\right|\dx \notag\\
         &\leq L_{\rm an} \int_D |m_\e(x)-m_0(x)|\dx \notag\\
         &\leq L_{\rm an} |D|^{1/2} \|m_\e-m_0\|_{L^2(D; \hspace{0.03cm} \R^3)}, \notag 
     \end{align}
     which in turn implies \eqref{eq8a}.\\
     
\underline{Step 2: The convergence of $\{\Ze_\e\}$.}
    The convergence of the Zeeman energies  follows directly via an analogous argument to Step 1.
    On the one hand, since $\{\Ms(T_{\cdot/\e}\tomega)m_0\}$ weakly converges to $\Ev[\Ms]m_0$ in $L^2(\domain)$,
    \begin{equation*}
        \Ze_\e(m_0) = -\mu_0 \int_D h_a\cdot \Ms(T_{x/\e}\tomega)m_0(x)\dx
        \overset{\e\rightarrow 0}{\longrightarrow}-\mu_0 \int_D h_a\cdot \Ev[\Ms]m_0(x)dx
        =\Zehom(m_0).
    \end{equation*}
    On the other hand,
    \begin{equation*}
        \left|\Ze_\e(m_\e) - \Ze_{\e}(m_0) \right|
        \leq \mu_0 C_{\rm sat} \|h_a\|_{L^2(\domain;\Rt)}\|m_\e-m_0\|_{L^2(\domain;\Rt)}.
    \end{equation*}
\end{proof}

\subsection{Equi-coerciveness of the micromagnetic functional}
\label{section:equicoerc}
The main part left for the proof of Theorem \ref{mainthm} is to show that the family $\{\F_\e\}$ is equi-coercive in the weak topology of $H^1(D, \St)$. 
This then guarantees the validity of the fundamental theorem of $\Gamma$-convergence concerning the variational convergence of minimum problems (see \cite{BD98, DM93}).

%Our aim is to investigate the effective behavior of energies $\F_\e$ via $\Gamma$-convergence for the weak topology of $H^1(D; \hspace{0.03cm}\St)$. The first step consists in proving the equi-coerciveness of the family $\F_\e$ with respect to the weak topology of  $H^1(D; \hspace{0.03cm} \St)$, which is the subject of the next proposition. 
\begin{prop}
\label{prop:equicoerciveness}
The family of functionals $\F_\e$ is mildly equi-coercive in the weak topology of $H^1(D;\hspace{0.03cm}\St)$. In other words, there exists a weakly compact set $K\subset H^1(D; \hspace{0.03cm} \St)$ such that 
     \begin{equation}
         \notag
         \inf_{m\in H^1(D; \hspace{0.03cm} \St)} \F_\e(m) = \inf_{m\in K} \F_\e(m) \qquad \mbox{ for every } \e>0.
     \end{equation}
\end{prop}

\begin{proof} Rewriting $\G_\e=\F_\e+\K_\e$ as in \eqref{def:funncEx+DMI} and applying the inequality $(a-b)^2\geq {1\over2}a^2- b^2$, along with assumption \ref{asPara_Ex}, yields
    \begin{align}
     \G_\e(m) &= {1\over 2}\int_D  a(T_{x/\e}\widetilde{\omega}) \left| \nabla m(x)   - 
 {\kappa(T_{x/ \e}\widetilde{\omega})\over a(T_{x/\e}\widetilde{\omega})}\bchi(m(x))\right|^2\dx  - {1\over 2}\int_D {\kappa^2(T_{x/ \e}\widetilde{\omega})\over a(T_{x/ \e}\widetilde{\omega})}|m(x)|^2\dx .\notag\\
 &\geq   {1\over 4} \int_D a(T_{x/ \e}\widetilde{\omega}) |\nabla m(x)|^2\dx - \int_D{\kappa^2(T_{x/\e}\widetilde{\omega})\over a(T_{x/ \e}\widetilde{\omega})}|m(x)|^2\dx  \notag\\
      &\geq {1\over 4}c_{\rm ex} \int_D|\nabla m(x)|^2\dx  -\int_D{\kappa^2(T_{x/ \e}\widetilde{\omega})\over a(T_{x/ \e}\widetilde{\omega})}|m(x)|^2\dx .\label{eq-c3}
    \end{align}
In view of assumptions \ref{asPara_Ex} and \ref{asPara_DMI} combined with an application of the Cauchy-Schwarz inequality, the last integral in \eqref{eq-c3} is bounded by
    \begin{equation}
    \left| \int_D{\kappa^2(T_{x/ \e}\widetilde{\omega})\over a(T_{x/ \e}\widetilde{\omega})}|m(x)|^2\dx \right|\leq {C_{\rm DMI}^2\over c_{\rm ex}}\|m\|_{L^2(D;\hspace{0.03cm} \St)} \leq C {C_{\rm DMI}^2\over c_{\rm ex}} |D|,\notag
   \end{equation}
where we have used the fact that $\|m\|_{L^2(D;\hspace{0.03cm} \St)}\leq C|D|$. Therefore,
    \begin{align}
      \G_\e(m)  &\geq {1\over 4}c_{\rm ex} \int_D|\nabla m(x)|^2\dx  -C {C_{\rm DMI}^2\over c_{\rm ex}}|D|.\label{eq-c4}
    \end{align}
The interaction energy $\Ze_\e$ is estimated by
   \begin{equation}
    \left|\mu_0 \int_Dh_a\cdot M_\e m(x)\dx  \right| \leq \mu_0\|h_a\|_{L^2(D)}\|M_\e m\|_{L^2(D)}\leq C\mu_0 |D||h_a|.\label{eq-c5}
    \end{equation}
Since the magnetostatic self-energy $\W_\e$, as well as the anisotropy energy $\A_\e$ are non-negative terms, and owing to \eqref{eq-c4} and \eqref{eq-c5}, we conclude that 
   \begin{align}
   \F_\e(m) & \geq {1\over 4}c_{\rm ex} \int_D|\nabla m(x)|^2\dx  - C {C_{\rm DMI}^2\over c_{\rm ex}}|D| - C\mu_0 |D||h_a|.\label{eq-c1}
   \end{align}
Now, set
    \begin{align}
    \notag
        \widetilde{\F}_\e(m):= \F_\e(m) +C \left(  {C_{\rm DMI}^2\over c_{\rm ex}} + \mu_0 |h_a| \right)|D|,
    \end{align}
so that, from inequality \eqref{eq-c1}, it follows that 
     \begin{equation}
         \label{eq-c2}
         C\|m\|_{H^1(D;\hspace{0.03cm}\St)}\leq  \widetilde{\F}_\e(m),
     \end{equation}
for every magnetization $m\in H^1(D; \hspace{0.03cm} \St)$. On the other hand, by assumptions \ref{asPara_Ex}, \ref{asPara_DMI}, and \ref{asPara_aniso}, we deduce that 
  \begin{align}
  \F_\e(m) %&= {1\over 2} \int_D a(T_{x/ \e}\tilde{\omega}) \left| \nabla m -{\kappa(T_{x\over \e}\tilde{\omega})\over a(T_{x\over \e}\tilde{\omega})}\bchi(m)\right|^2 - \int_D {\kappa^2(T_{x\over \e}\tilde{\omega})\over a(T_{x\over \e}\tilde{\omega})}|m|^2\dx  -{\mu_0\over 2} \int_D h_d[M_\e m]\cdot M_\e m\dx    	\notag\\
  %& \quad +\int_D \varphi(T_{{x\over\e}}\omega, m)\dx  -\mu_0 \int_Dh_a\cdot M_\e m\dx \notag\\
  &\leq {1\over 2} \int_D a(T_{x/ \e}\widetilde{\omega}) |\nabla m(x)|^2\dx  + {1\over 2} \int_D {\kappa^2(T_{x/\e}\widetilde{\omega})\over a(T_{x/ \e}\widetilde{\omega})}|m(x)|^2\dx  + {\mu_0\over 2}\int_D|h_a[M_\e(x) m(x)]|^2\dx  \notag\\
  &\quad + C_{\rm an} |D|+ C\mu_0|D||h_a|C_s\notag\\
  &\leq {C_{\rm ex}\over 2} \int_D|\nabla m(x)|^2\dx  +{C^2_{\rm DMI}\over 2c_{\rm ex}} \|m\|^2_{L^2(D;\hspace{0.03cm} \M)} +C,\notag
  \end{align}
where we used \eqref{stabilityesti} for the second inequality.
In particular, there exists a positive constant  $\widetilde{C}$ such that $\widetilde{\F}_\e(m)\leq\widetilde{C}\|m\|^2_{H^1(D; \hspace{0.03cm} \St)}$
Therefore, for every magnetization $m\in H^1(D; \hspace{0.03cm} \St)$, there holds $\widetilde{\F}_\e(m)\leq C|D|$. 
Let $K(D; \hspace{0.03cm} \St)$ be the set defined by
    \begin{equation}
    \notag
    K(D; \hspace{0.03cm} \St):= \{m\in H^1(D; \hspace{0.03cm} \St)\hspace{0.03cm} :\hspace{0.03cm} \widetilde{\F}_\e(m)\leq C|D| \}.
    \end{equation}
Note that in view of \eqref{eq-c2}, if the magnetization $m\in K(D; \hspace{0.03cm} \St)$, then $\|\nabla m\|_{L^2(D; \hspace{0.03cm} \St)}\leq \widetilde{C}|D|$. Hence, $K(D; \hspace{0.03cm} \St)$ is contained in a ball $B_{\St}$ of $H^1(D; \hspace{0.03cm} \R^3)$. Setting $K:=K(D; \hspace{0.03cm} \St):= B_{\St}\cap H^1(D; \hspace{0.03cm} \St)$, we conclude that
  \begin{equation}
  \notag
 \inf_{m\in H^1(D; \hspace{0.03cm} \St)} \widetilde{\F}_\e(m)=\inf_{m\in K} \widetilde{\F}_\e(m),
  \end{equation}
where $K$ is weakly compact being the intersection of the weakly closed set $H^1(D; \hspace{0.03cm} \St)$  and the weakly compact set $B_{\St}$. This concludes the proof.
\end{proof} 

\subsection{Limiting behavior of the micromagnetic functionals - the proof of Theorem \ref{mainthm}}
%\EEE E: here we should add a quick proof of the full gamma-convergence result \BBB

%In this section we deal with the computation of the limit of $\F_\e$. In view of the analysis in the previous subsections, the magnetocrystalline asisotropy energy $\A_\e$ and the Zeeman energy $\Ze_\e$ as well as the magnetostatic energies $\W_\e$ turn out to be $\Gamma$-continuous perturbations of the micromagnetic energy functional $\F_\e$. This combined with \cite[Proposition 6.20]{DM93} implies that 
%      \begin{align}
%      \notag
 %         \Gamma\mbox{-}\lim_{\e\to 0} \F_\e = \Gamma\mbox{-}\lim_{\e\to 0} (\E_\e + \K_\e) + \Whom + \Ahom + \Zehom.
  %    \end{align}

We are now in a position to prove Theorem \ref{mainthm}. First, the equi-coerciveness of the family $\F_\e$ with respect to the weak topology $H^1(D;\hspace{0.03cm} \St)$ is proven in Section \ref{section:equicoerc}. In view of the stability property of  $\Gamma$-limits under the sum with a continuously convergent family of functionals (see \cite[Proposition 6.20]{DM93}), we deduce that, for $m\in L^2(D;\hspace{0.03cm}\St)$, 
       \begin{align}
      \notag
          F_\hom(m) & = \Gamma\mbox{-}\lim_{\e\to 0} \F_\e(m) = \Gamma\mbox{-}\lim_{\e\to 0} (\E_\e + \K_\e) (m) + \Whom(m) + \Ahom (m)+ \Zehom(m)\\
          &=\Ghom(m) + \Whom(m) + \Ahom (m)+ \Zehom(m),\notag
      \end{align}
where  $\Ghom$ is as in \eqref{eq:intro:Ghom},  $\Whom$ is defined by \eqref{eq:intro:def:Whom}, $\Ahom$ and $\Zehom$ are given, respectively, by \eqref{eq:intro:def:Ahom} and \eqref{eq:intro:def:Zehom}. This completes the proof of Theorem \ref{mainthm}.

\section{Application to Multilayers}
\label{section:multilayers}

\begin{figure}[t]
\centering
    \begin{minipage}{0.54\textwidth}
        \centering
        \includegraphics[width=4cm]{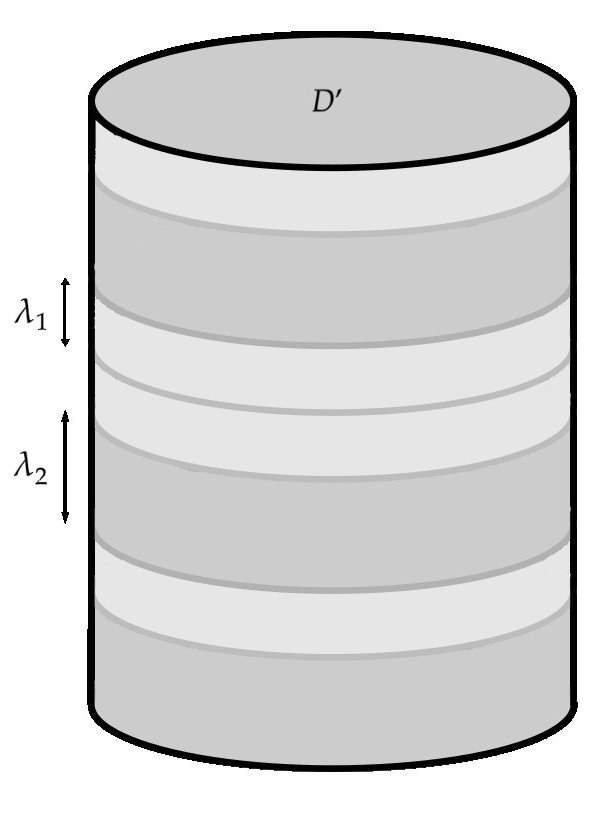}
        \end{minipage}
      \begin{minipage}{0.45\textwidth}  
        \captionof{figure}{We consider a chiral multilayer having a random laminated structure. E.g. we assume here that there are two possible layer ratios $\lambda_1, \lambda_2\in(0,1)$ both occurring with a certain probability.}
        \label{fig:randomlam}
    \end{minipage}
 \end{figure}

In this section, we specify the effective energy $\Ghom$ from Theorem \ref{thm:GeGammaConv} in the micromagnetic setting of chiral multilayers. 
In such a case, $\M=\mathbb{S}^2$ and $D\subset\R^3$ has a laminated structure. 
More specifically, we assume that $D\coloneqq D'\times I_\lambda$, where $D'$ is a bounded open subset of $\R^2$ with $|D'|=1$ and $I_\lambda=(0, \lambda)$, for some $\lambda>0$ (see Figure \ref{fig:randomlam}).
In addition to the usual assumptions, we let the dynamical system satisfy
   \begin{equation}\label{eq:lami:T1dim}
       T_{(x_1, x_2, x_3)}= T_{(0, 0, x_3)} = T_{x_3 e_3}
       %=T_{x_3}
       \quad\text{for all $x\in\Rt$}.
   \end{equation}
In other words, the ergodic action described by $T$ does not depend on the variable $(x_1, x_2)\in D'$.

\begin{remark}[Differentiability in the multilayer setting]\label{rem:lami:diff}
    In particular, any function $u\in L^2(\Omega)$ is everywhere differentiable in the directions $e_1,e_2$ with $D_1 u= D_2 u =0$ in the sense of Definition \ref{def:derivates}.
    Therefore, the first two rows of every function in $\Ltpot(\Omega;\Rt)$ are $0$ almost everywhere. This holds, in particular, for $\Theta_a,\Theta_\kappa\in\Ltpot(\Omega;\Rt)$ as defined in Lemma \ref{lem:correctors}.
\end{remark}

%namely, for any $u\in L^2(\Omega; \hspace{0.03cm} \R^3)$, $\omega\in\Omega$ and $(x_1, x_2)\in D'$, $u(T_{(x_1, x_2, 0)}\omega)= u(\omega)$ and $D_{x_1}u=D_{x_2}u =0$. Hence, any $N\in \Ltpot(\Omega; \hspace{0.03cm} \R^{3\times 3})$ has the first two columns which vanish. In particular, $\Theta_a$ and $\Theta_\kappa$ given by Proposition \ref{prop:homDensityDescr} are of such a form.

\begin{prop}[Characterization of $\Ghom$ given a laminated structure]\label{prop:lami:Ghom}
Under the assumptions of Theorem \ref{mainthm}, and assuming additionally that the dynamical system satisfies \eqref{eq:lami:T1dim},
the homogenized energy functional $\Ghom$ as defined in \eqref{eq:intro:Ghom} is given by
\begin{equation}\label{eq:lami:Ghom}
    \Ghom(m) 
    =   \frac{1}{2}\int_\domain \nabla m : a_\eff \nabla m \dx
    -   \int_\domain \nabla m : \kappa_\eff \bchi(m) \dx
    +   \int_\domain \vphi_\eff (m) \dx.
\end{equation}
for $m\in H^1(D; \hspace{0.03cm} \St)$ with constants
\begin{align}
    a_\eff          &\coloneqq  \operatorname{diag}\left(\Ev[a],\Ev[a],\Ev\left[a^{-1}\right]^{-1}\right)\in\Rtt,                                           \notag\\
    \kappa_\eff     &\coloneqq  \operatorname{diag}\left(\Ev[\kappa],\Ev[\kappa],\Ev\left[\frac{\kappa}{a}\right]\Ev\left[a^{-1}\right]^{-1}\right)\in\Rtt, \notag\\
\intertext{and effective anisotropy density $\vphi_\eff\in C^1(\St,\R)$ defined as}
    s\mapsto \vphi_\eff(s)   
                    &\coloneqq  -\frac{1}{2}\left(\Ev\left[\frac{\kappa^2}{a}\right]-\Ev\left[\frac{\kappa}{a}\right]^2\Ev[a^{-1}]^{-1}\right)\left(s_1^2+s_2^2\right). \notag
\end{align}
\end{prop}

\begin{remark}[Comparison to the periodic case]\label{rem:lami:compPeriodic}
    It is easy to see that this is the same energy which Di Fratta and the first author obtained for periodic multilayers \cite[Theorem 3.1]{DDF20}, only due to the setting the expected values are replaced with cell averages.
\end{remark}

\begin{proof}
    We provide explicit characterizations of the correctors $\Theta_a= \left(\Theta_{a.ij}\right)_{i,j=1,2,3} $ $ \Theta_\kappa = \left(\Theta_{\kappa.ij}\right)_{i,j=1,2,3}$ in this setting. 
    We then plug these characterizations into the definition of $\Ghom$, cf \eqref{eq:intro:Ghom}, to obtain the result.

\underline{Step 1: The characterization of $\Theta_a$.}
    With Remark \ref{rem:lami:diff}, it remains to identify $\Theta_{a,3i}$ for $i=1,2,3$, which are defined by
    \begin{equation}\label{eq:lami:Theata_ij}
        a\Theta_{a,3i}e_3 + a e_i \in\Ltsol(\Omega).
    \end{equation}
    Again from Remark \ref{rem:lami:diff}, it immediately follows that $a e_1,a e_2\in\Ltsol(\Omega)$ and hence $\Theta_{a,31} = \Theta_{a,32} = 0$.
    For $i=3$ we rewrite \eqref{eq:lami:Theata_ij} as 
    \begin{equation}
    \notag
        \int_\Omega (a(\omega)\Theta_{a,33}(\omega) + a(\omega)) D_3 b(\omega)\dPo 
        = 0
        \quad\text{for all }b\in C^1(\Omega).
    \end{equation}
    By Definition \ref{def:weakderivative}, we thus have $D_3(a\Theta_{a,33}+a)=0$ and hence $a\Theta_{a,33}+a\in H^1(\Omega)$ with $\nablaomega(a\Theta_{a,33}+a)=0$ due to Remark \ref{rem:lami:diff}.
    Thus, by Proposition \ref{prop:uconstant} there exists $C_a\in\R$ such that $\P$-almost surely
    \begin{equation*}
        a\Theta_{a,33}+a = C_a.
    \end{equation*}
    We extend Lemma \ref{lemma:stochintegrationparts} to components of $\Ltpot(\Omega)$-functions via approximation, and have
    \begin{equation*}
        0 = \Ev[\Theta_{a,33}] = \Ev\left[a^{-1}\right]{C_a}-1
    \end{equation*}
    Hence, we obtain $C_a=\Ev[a^{-1}]^{-1}$ and therefore 
    \begin{equation}\label{eq:lami:Thetaa}
        \Theta_{a,33}=\frac{\Ev[a^{-1}]^{-1}}{a}-1.
    \end{equation}

\underline{Step 2: The characterization of $\Theta_\kappa$.}
    As in Step 1, it holds that $\Theta_{\kappa,ij}=0$ for $(i,j)\neq (3,3)$ and there exists a constant $C_\kappa\in\R$, such that $\P$-almost surely
    \begin{equation*}
        a\Theta_{\kappa,33}+\kappa = C_\kappa.
    \end{equation*}
    Analogously, we obtain $C_\kappa = \Ev[\kappa/a]\Ev[a^{-1}]^{-1}$ and therefore
    \begin{equation}\label{eq:lami:Thetak}
        \Theta_{\kappa,33}=\Ev\left[\frac{\kappa}{a}\right]\frac{\Ev[a^{-1}]^{-1}}{a}-\frac{\kappa}{a}.
    \end{equation}

\underline{Step 3: Plugging $\Theta_a,\Theta_\kappa$ into \eqref{eq:intro:Ghom}.}
    With \eqref{eq:lami:Thetaa} we have
    \begin{equation*}
        \Ev\left[a\Theta_a^\intercal\Theta_a\right]
        =
        \Ev\left[a^{-1}\left(\Ev[a^{-1}]^{-1}-a\right)^2\right]e_3 \otimes e_3
        =
        \left(\Ev[a]-\Ev[a^{-1}]^{-1}\right)e_3 \otimes e_3.
    \end{equation*}
    Additionally considering \eqref{eq:lami:Thetak}, for the mixed term we obtain
    \begin{align*}
        \Ev\left[a\Theta_a^\intercal\Theta_\kappa\right]
        &=
        \Ev\left[\left(\Ev[a^{-1}]^{-1}-a\right)\left(\Ev\left[\frac{\kappa}{a}\right]\frac{\Ev[a^{-1}]^{-1}}{a}-\frac{\kappa}{a}\right)\right]e_3 \otimes e_3\\
        &=
        \left(\Ev[\kappa]-\Ev\left[\frac{\kappa}{a}\right]\Ev[a^{-1}]^{-1}\right)e_3 \otimes e_3.
    \end{align*}
    For the effective anisotropy energy term, \eqref{eq:lami:Thetak} yields 
    \begin{align*}
        \Ev\left[a\Theta_\kappa^\intercal\Theta_\kappa\right]
        &=
        \Ev\left[a^{-1}\left(\Ev\left[\frac{\kappa}{a}\right]\Ev[a^{-1}]^{-1}-\kappa\right)^2\right]e_3 \otimes e_3\\
        &=
        \left(\Ev\left[\frac{\kappa^2}{a}\right]-\Ev\left[\frac{\kappa}{a}\right]^2\Ev[a^{-1}]^{-1}\right)e_3 \otimes e_3.
    \end{align*}
    Additionally, note that via direct calculation
    \begin{equation*}
        \bchi(m):(e_3\otimes e_3)\bchi(m) = \left|e_3\times m\right|^2= m_1^2 + m_2^2.
    \end{equation*}
    Plugging these identities into \eqref{eq:intro:Ghom} yields \eqref{eq:lami:Ghom}.
\end{proof}

Due to Remark \ref{rem:lami:compPeriodic}, via a result from \cite{DDF20} we immediately obtain a characterization for the minimizers of $\Ghom$ given $\Ev[\kappa]=0$.

\begin{lemma}[Characterization of minimizers, {{\cite[Theorem 3.1]{DDF20}}}]\label{lem:lami:minimizers}
    In the setting of Proposition \ref{prop:lami:Ghom},  additionally assuming that $\Ev[\kappa]=0$, the only minimizers in $H^1(D; \hspace{0.03cm} \St)$ of $\Ghom$ given by \eqref{eq:lami:Ghom} are helical textures $m_*$ of the form
    \begin{align*}
        m_*(x)      &\coloneqq \cos(\theta(x\cdot e_3))e_1 + \sin(\theta(x\cdot e_3))e_2,\\
        \theta(t)   &\coloneqq \theta_0 + \Ev\left[\frac{\kappa}{a}\right]t,
        \quad\text{for }t\in\R,
    \end{align*}
    with arbitrary $\theta_0\in\R$.
    The minimum value of the energy is given by
    \begin{equation*}
        \Ghom(m_*) = -\frac{\lambda}{2} \Ev\left[\frac{\kappa^2}{a}\right].
    \end{equation*}
\end{lemma}

\section*{Acknowledgements}
All authors acknowledge support of the Austrian Science Fund (FWF) through the SFB project F65. The research of E. Davoli and L. D'Elia has additionally been supported by the FWF through grants V662, Y1292, and P35359.

\section*{Data availability statement}
Data sharing not applicable to this article as no datasets were generated or analysed during the current study.

\end{document}